\documentclass[11pt]{amsart}
\usepackage{amssymb,graphics,color,graphicx}
\usepackage[all]{xy}

\title[Crystal energy via charge]{Crystal energy functions via the charge\\ in types $A$ and $C$}

\author[C.~Lenart]{Cristian Lenart}
\address{Department of Mathematics and Statistics, State University of New York at Albany, 
Albany, NY 12222, U.S.A.}
\email{lenart@albany.edu}
\urladdr{http://www.albany.edu/\~{}lenart/}
\thanks{Partially supported by the Research in Pairs program of Mathematiches Forschungsinstitut Oberwolfach, the NSF grant DMS--1101264, and an Individual Development Award from SUNY Albany.}

\author[A.~Schilling]{Anne Schilling}
\address{Department of Mathematics, University of California, One Shields
Avenue, Davis, CA 95616-8633, U.S.A.}
\email{anne@math.ucdavis.edu}
\urladdr{http://www.math.ucdavis.edu/\~{}anne}
\thanks{Partially supported by the NSF grants DMS--0652641, DMS--0652652, DMS--1001256, the ``Research in Pairs'' program 
by the Mathematiches Forschungsinstitut Oberwolfach in 2011, and the Hausdorff Institut in Bonn.}
 


\DeclareMathOperator{\lev}{lev}

\newlength{\cellsize}
\newcommand\tableau[1]{
\vcenter{
\let\\=\cr
\baselineskip=-16000pt
\lineskiplimit=16000pt
\lineskip=0pt
\halign{&\tableaucell{##}\cr#1\crcr}}}

\newcommand{\tableaucell}[1]{{%
\def \arg{#1}\def \void{}%
\ifx \void \arg
\vbox to \cellsize{\vfil \hrule width \cellsize height 0pt}%
\else
\unitlength=\cellsize
\begin{picture}(1,1)
\put(0,0){\makebox(1,1){$#1$}}
\put(0,0){\line(1,0){1}}
\put(0,1){\line(1,0){1}}
\put(0,0){\line(0,1){1}}
\put(1,0){\line(0,1){1}}
\end{picture}%
\fi}}

\numberwithin{equation}{section}

\theoremstyle{plain}
\newtheorem{theorem}{Theorem}[section]
\newtheorem{proposition}[theorem]{Proposition}
\newtheorem{lemma}[theorem]{Lemma}
\newtheorem{corollary}[theorem]{Corollary}

\newtheorem{definition}[theorem]{Definition}

\newtheorem{example}[theorem]{Example}

\newtheorem{algorithm}[theorem]{Algorithm}
\newtheorem{arule}[theorem]{Rule}

\theoremstyle{remark}
\newtheorem{remark}[theorem]{Remark}

\def\Z{\mathbb{Z}}

\def\charge{{\rm charge}}
\def\cw{{\rm cw}}

\def\wt{{\rm wt}}

\def\g{\mathfrak{g}}

\begin{document}

\bibliographystyle{plain}

\begin{abstract} 
The Ram--Yip formula for Macdonald polynomials (at $t=0$) provides a statistic which we call charge.
In types $A$ and $C$ it can be defined on tensor products of Kashiwara--Nakashima single column crystals.
In this paper we prove that the charge is equal to the (negative of the) energy function on affine
crystals. The algorithm for computing charge is much simpler and can be more efficiently computed 
than the recursive definition of energy in terms of the combinatorial $R$-matrix. 
\end{abstract}

\maketitle

%%%%%%%%%%%%%%%%%%%%%%%%%%%%%%%%%%%%%%%%%%%%%%%%%%% 
\section{Introduction}

The energy function of affine crystals is an important grading used in one-dimensional configuration 
sums~\cite{HKOTT:2002, HKOTY:1999} and generalized Kostka polynomials~\cite{SW:1999,Sh:2001,Sh:2002}. 
It is defined by the action of the affine Kashiwara crystal operators through a local combinatorial rule and the 
$R$-matrix.

From a computational perspective, the definition of the energy is not very efficient,
as it involves a recursive definition of a local energy, and also the combinatorial $R$-matrix, for which not in
all cases efficient algorithms exist. This leads us to the role of the charge statistic, which can be calculated very 
efficiently, as it only involves the detection of descents and the computation of arm lengths of cells in Young diagrams.

Charge was originally defined in type $A$ by Lascoux and Sch\"utzenberger~\cite{LS:1979} as a statistic on 
words with partition content. It is calculated by counting certain cycles in the given word; see 
Section~\ref{section.chargerev}. Lascoux and Sch\"utzenberger showed that the charge can also be defined 
as the grading in the cyclage graph, and used it to express combinatorially the Kostka--Foulkes 
polynomials, or Lusztig's $q$-analogue of weight multiplicities \cite{lusscq}, based on their Morris recurrence. In type $A$, 
Nakayashiki and Yamada~\cite{NY:1997} analyzed the subtle combinatorial relationship between charge and 
the $R$-matrix, showing that the energy coincides with the charge. In~\cite{kmopdc} it was observed that the 
cyclage is related to the action of the crystal operator $f_0$ on a tensor product of type $A$ columns (by the 
Kyoto path model, the latter can be identified with an affine Demazure crystal). Thus, the results of 
Lascoux--Sch\"utzenberger and Nakayashiki--Yamada are rederived in a more conceptual way. See also the
work of Shimozono \cite{Sh:2001,Sh:2002} for a more extensive discussion of the combinatorics involved 
in~\cite{kmopdc}, in the more general context of tensor products of type $A$ Kirillov--Reshetikhin (KR) crystals of 
arbitrary rectangular shapes, as opposed to only column shapes. Charge for KR crystals of rectangular
shape (or Littlewood--Richardson tableaux) was also defined in~\cite{SW:1999} using cyclage.

Lecouvey~\cite{leckfp,lecccg} extended two approaches to the Lascoux--Sch\"utzenberger charge, namely 
cyclage and catabolism, to types $B$, $C$, and $D$. He thus defined a charge statistic 
on the corresponding Kashiwara--Nakashima (KN) tableaux~\cite{kancgr}. But he was only able to relate his 
charge to the corresponding Kostka--Foulkes polynomials in very special cases, as the original idea of 
Lascoux--Sch\"utzenberger based on the Morris recurrence, which he pursued, has limited applicability in this case.

In this paper we use a charge statistic coming from the Ram--Yip formula~\cite{raycfm} for Macdonald polynomials 
$P_\mu(x;q,t)$ of arbitrary type~\cite{macaha} at $t=0$. The terms in this formula correspond to certain chains of Weyl group 
elements which come from the alcove walk model (this was defined in \cite{gallsg,lapawg,lapcmc}, and was then developed in
subsequent papers, including~\cite{raycfm}). The statistic is defined on the mentioned chains, and describes the powers 
of $q$. In~\cite{lenmpc} it is shown that, in types $A$ and $C$, the chains are in bijection with elements in a tensor product 
of KR crystals of the form $B^{k,1}$. It is also shown that, under this bijection, the above statistic can be translated 
into a statistic on the elements of the mentioned crystal, which we call charge. Thus, we have
\begin{equation}\label{p1}
	P_\mu(x;q,0) = \sum_{b \in B^{\mu_1',1}\otimes B^{\mu_2',1}\otimes\ldots} 
	q^{\charge(b)} x^{\wt(b)}.
\end{equation}
In type $A$, one can rewrite this formula as an expansion of the Macdonald $P$-polynomials in terms of 
Schur polynomials $s_\lambda(x)$
\begin{equation}\label{p2}
	P_\mu(x;q,0) = \sum_{\lambda} K_{\lambda' \mu'}(q) \; s_\lambda(x)\,;
\end{equation}
here $K_{\lambda'\mu'}(q)$ is the Kostka--Foulkes polynomial and $\lambda'$ denotes the transpose of 
the partition $\lambda$. A generalization of \eqref{p2} to simply-laced types was given in~\cite{ionnmp}; in types $A$ and $D$, 
this result is sharpened in \cite[Section 9.2]{ST:2011} by replacing the Kostka--Foulkes polynomials with the corresponding 
one-dimensional configuration sums (which are generating functions for the energy). 
Both \eqref{p1} in type $A$ and~\eqref{p2} are expressed combinatorially in terms of the 
Lascoux--Sch\"utzenberger charge, whereas the type $C$ charge given by~\eqref{p1} is a new statistic. It is worth noting that the 
main ingredient in these charge constructions is the quantum Bruhat graph~\cite{bfpmbo}, which first arose in 
connection to Chevalley multiplication formulas for the quantum cohomology of flag varieties.

Due to Schur--Weyl duality, weight multiplicities and tensor product multiplicities are both given by the Kostka matrix in type $A$.
In fact, the Kostka--Foulkes polynomials are equal to Lusztig's $q$-analogue $K$ of weight multiplicities as well as
the one-dimensional configurations sums $X$, which are $q$-analogues of tensor product multiplicities. For other types,
such correspondences are only known for large rank (or in the stable case); see for
example~\cite{Lecouvey:2006a,LOS:2010,SZ:2006}. 
As is evident from Equation~\eqref{p1} the charge we are concerned with in this paper is naturally related to $q$-analogues of 
tensor product multiplicities or the one-dimensional sums $X$.

The goal of this paper is to show in an efficient, conceptual way that the charge in~\cite{lenmpc} coincides with 
the energy function on the corresponding tensor products of KR crystals.  We focus on types $A$ and $C$, and expect to 
extend these results to types $B$ and $D$. There are several reasons for not yet addressing the case of arbitrary types. 
First of all, the elements of the crystals $B^{k,1}$ of classical types are represented in a very explicit way, by KN columns. 
In fact, concrete descriptions of certain KR crystals of exceptional type have only been given in special cases; see for
example~\cite{jasast, KMOY:2007, MMO:2010, Yamane:1997}.
Furthermore, the stable one-dimensional configuration sums studied in~\cite{LOS:2010} are all of classical type. 
However, we do expect that the main result of this paper would generalize to arbitrary types, if we were to phrase it in the 
context of the alcove walk model and the statistic in the Ram--Yip formula mentioned above (see Section~\ref{generalize}).

We use the recent reinterpretation in~\cite{ST:2011} of the (global) energy function as the affine grading on a tensor 
product of KR crystals under Demazure arrows (see Definition~\ref{definition.demazure_arrows}).
In type $A$, KR crystals are perfect and hence, by the Kyoto path model~\cite{KMN1:1992, KMN2:1992},
can be realized as Demazure crystals. By the result of~\cite{ST:2011}, Demazure arrows change energy by 1. 
Together with the result that charge is 
well-behaved under crystal operators, this proves the equality of energy and charge.
For type $C$, we use the same approach, but in this case KR crystals are not perfect.
There is still an embedding of a tensor product of KR crystals into an affine highest weight crystal
(see Proposition~\ref{proposition.expansion_nonperfect}) by analogy with the Kyoto path model,
but now there are several highest weight components in the image, instead of just one. For each of these 
components we exhibit an explicit path from its highest weight (or ground state) to ``type $A$ elements'' in the 
component, using only  Demazure arrows (see Theorem~\ref{theorem.ground2typeA}). This additional result 
suffices to establish the equality of energy and charge in type $C$, based on the corresponding result in 
type $A$. As a by-product, we obtain an explicit description of the components that appear in the nonperfect 
setting of single columns for type $C$.

Our main result can now be stated as follows.

\begin{theorem} \label{theorem.main}
Let $B=B^{r_N,1} \otimes \cdots \otimes B^{r_1,1}$ be a tensor product of KR crystals in type 
$A_{n-1}^{(1)}$ or type $C_n^{(1)}$ with $r_N\ge r_{N-1}\ge \cdots \ge r_1 \ge 1$. Then for all $b\in B$ we have
\begin{equation}
	D(b) = - \charge(b),
\end{equation}
where $D(b)$ is given in Definition~{\rm \ref{definition.D_B}}, and $\charge(b)$ is defined in 
Section~{\rm \ref{section.chargerev}}.
\end{theorem}

From a theoretical point of view, the above result is not surprising due to work of Ion~\cite{ionnmp}, which relates 
Macdonald polynomials at $t=0$ and affine Demazure characters in simply-laced types. However, this result 
does not work in type $C$; in addition, it only gives an equality of polynomials (the generating functions for 
the statistics and the weights), not of individual terms.

To compare our work with the previous papers on charge and the energy, let us first say that our results apply to arbitrary 
vertices in a tensor product of KN columns, not just to the highest weight elements (with respect to the nonzero arrows), 
that are used in the work involving Kostka--Foulkes polynomials. In type $A$, our approach via affine Demazure 
crystals comes closest to~\cite{kmopdc,Sh:2001,Sh:2002}. However, in addition to the aspect mentioned above, it  
differs from the approach in these papers because we do not use the cyclage operation, which is based 
on the corresponding plactic relations; see~\cite{leckfp}. These relations are the main cause of the complications in 
type $C$, in the work of Lecouvey~\cite{leckfp,lecccg}.

The paper is organized as follows. In Section~\ref{section.energy} we review the necessary crystal theory
and define the energy function. In Section~\ref{section.chargerev} we give the definition of charge both
in types $A$ and $C$. The proof of Theorem~\ref{theorem.main} for type $A$ using the method
of Demazure arrows is given in Section~\ref{section.energyA}. In Section~\ref{section.kyotoC} we classify
the various components under the Demazure arrows in type $C$ in the extension of the Kyoto path model,
and show that each ground state is connected to a type $A$ filling. These results are used in 
Section~\ref{section.energyC} to prove Theorem~\ref{theorem.main} for type $C$.
We conclude in Section~\ref{section.outlook} with a discussion of various possible extensions of this work.

%%%%%%%%%%%%%%%%%%%%%%%%%%%%%%%%%%%%%%%%%%%%%%%%%%% 
\subsection*{Acknowledgements}
We would like to thank the Mathematisches Forschungsinstitut in Oberwolfach (MFO), Germany, for their support.
The breakthrough in this project occurred during Research in Pairs at MFO in February 2011.
Many of our results were found exploring with code on crystals in {\sc Sage}~\cite{sage,sagecombinat}.

We would also like to thank S.~Gaussent, M.~Kashiwara, C.~Lecouvey, P.~Littelmann, S.~Morier-Genoud,
M.~Okado, and P.~Tingley for helpful discussions.

%%%%%%%%%%%%%%%%%%%%%%%%%%%%%%%%%%%%%%%%%%%%%%%%%%% 
\section{Crystals and energy function}
\label{section.energy}

In this section we review and set up crystal theory and define the energy function.
 
%%%%%%%%%%%%%%%%%%%%%%%%%%%%%%%%%%%%%%%%%%%%%%%%%%%
\subsection{Crystal generalities}
Crystal bases provide a combinatorial method to study representations of quantum algebras
$U_q(\g)$. For a good review on crystal base theory see the book by Hong and Kang~\cite{HK:2002}.
Here $\g$ is a Lie algebra or affine Kac--Moody Lie algebra with index set $I$, weight lattice $P$,
and simple roots $\alpha_i$ with $i\in I$. The set of dominant weights is denoted by $P_+$.
For affine Kac--Moody (resp. finite Lie) algebras we denoted the fundamental weights by $\Lambda_i$ 
(resp. $\omega_i$) for $i\in I$.

A $\g$-{\em crystal} is a nonempty set $B$ together with maps $e_i,f_i:B\to B\cup \{\emptyset\}$ for $i\in I$ and 
$\wt:B \to P$. For $b\in B$, we set $\varepsilon_i(b) = \max\{k \mid e_i^k(b) \neq \emptyset \}$,
$\varphi_i(b) = \max\{k \mid f_i^k(b) \neq \emptyset \}$,
\[
	\varepsilon(b) = \sum_{i\in I} \varepsilon_i(b) \Lambda_i \quad \text{and} \quad
	\varphi(b) = \sum_{i\in I} \varphi_i(b) \Lambda_i.
\]
The beauty about crystal theory is that it is well-behaved with respect to taking
tensor products. Let $B_1$ and $B_2$ be two $\g$-crystals. As a set $B_1\otimes B_2$ is the Cartesian
product of the two sets. For $b=b_1 \otimes b_2\in B_1 \otimes B_2$, the weight function is simply
$\wt(b) = \wt(b_1) + \wt(b_2)$. The crystal operators are given by
\begin{equation*}
	f_i (b_1 \otimes b_2)=
	\begin{cases}
	f_i  (b_1) \otimes b_2 & \text{if $\varepsilon _i(b_1) \geq \varphi_i(b_2)$,}\\
	b_1 \otimes f_i (b_2)  & \text{otherwise,}
	\end{cases}
\end{equation*}
and similarly for $e_i(b)$. This rule can also be expressed combinatorially by the signature rule.

A {\em highest weight crystal} $B(\lambda)$ of highest weight $\lambda\in P_+$ is a crystal with a unique element 
$u_\lambda$ such that $e_i(u_\lambda)=\emptyset$ for all $i\in I$ and $\wt(u_\lambda)=\lambda$.
On finite-dimensional highest weight crystals $B(\lambda)$ there exists an involution $S:B(\lambda) \to B(\lambda)$,
called the \textit{Lusztig involution}, which is a crystal isomorphism such that
\begin{equation*}
	S(f_i) = e_{i^*} \quad \text{and} \quad S(e_i) = f_{i^*}.
\end{equation*}
Here $i^*$ is defined through the map $\alpha_i \mapsto \alpha_{i^*} : = - w_0(\alpha_i)$ with
$w_0$ the longest element in the Weyl group of $\g$. Explicitly, we have $i^*=n-i$ for type $A_{n-1}$
and $i^* = i$ for type $C_n$. Under $S$ the highest weight element goes to the lowest weight element.

In~\cite{HK:2006}, Henriques and Kamnitzer defined a {\em crystal commutor} on the tensor product of
two classically highest weight crystals in terms of Lusztig's involution as
\begin{equation}
\label{equation.commutor}
\begin{split}
	B(\lambda) \otimes B(\mu) &\to B(\mu) \otimes B(\lambda)\\
	b_1 \otimes b_2 & \mapsto S(S(b_2) \otimes S(b_1)).
\end{split}
\end{equation}

%%%%%%%%%%%%%%%%%%%%%%%%%%%%%%%%%%%%%%%%%%%%%%%%%%% 
\subsection{Kashiwara--Nakashima columns for type $C$}\label{sectkn} 
Kashiwara and Nakashima~\cite{kancgr} developed general tableaux models for finite-dimensional
highest weight crystals for all non-exceptional classical Lie algebras $\g$.
For type $C_n$, the Kashiwara--Nakashima (KN) columns \cite{kancgr} of height $k$ index the vertices 
of the fundamental representation $V(\omega_k)$ of the symplectic algebra $\mathfrak{sp}_{2n}({\mathbb C})$.
These columns are filled with entries in $[\overline{n}]:=\{1<2<\cdots<n<\overline{n}<\overline{n-1}<\cdots<\overline{1}\}$.

\begin{definition} \label{definition.KN_column}
A column-strict filling $b=b(1)\ldots b(k)$ with entries in $[\overline{n}]$ is a KN column if there
is no pair $(z,\overline{z})$ of letters in $b$ such that: 
\[
	z = b(p)\,, \qquad \overline{z} = b(q)\,,\qquad q-p\le k - z\,.
\]
\end{definition}

We will need a different definition of KN columns, which was proved to 
be equivalent to the one above in~\cite{shesjt}.

\begin{definition}\label{defkn}
Let $b$ be a column and $I=\{z_1 > \cdots > z_r\}$ the set of unbarred letters $z$ such that
the pair $(z,\overline{z})$ occurs in $b$. The column $b$ can be split when there exists a set
of $r$ unbarred letters $J=\{t_1 > \cdots > t_r\} \subset[{n}]$ such that:
\begin{itemize}
\item $t_1$ is the greatest letter in $[n]$ satisfying: $t_1 < z_1$, $t_1\not\in b$, and $\overline{t_1}\not\in b$,
\item for $i = 2, \ldots, r$, the letter $t_i$ is the greatest one in $[n]$ satisfying $t_i < \min(t_{i-1},z_i)$, $t_i\not\in b$, 
and $\overline{t_i}\not\in b$.
\end{itemize}
In this case we write:
\begin{itemize}
\item $b^R$ for the column obtained by changing  $\overline{z_i}$ into $\overline{t_i}$ in $b$ for each 
letter $z_i\in I$, and by reordering if necessary,
\item $b^L$ for the column obtained by changing $z_i$ into $t_i$ in $b$ for each letter $z_i\in I$, and by 
reordering if necessary.
\end{itemize}
The pair $b^L b^R$ will be called a split column.
\end{definition}

\begin{example}
{\rm
The following is a KN column of height $5$ in type $C_n$ for $n\ge 5$, together with the corresponding 
split column:
\[
	b=\tableau{{4}\\{5}\\{\overline{5}}\\{\overline{4}}\\{\overline{3}}}\,,\;\;\;\;\;
	b^L b^R=\tableau{{1}&{4}\\{2}&{5}\\{\overline{5}}&{\overline{3}}\\{\overline{4}}&{\overline{2}}\\
	{\overline{3}}&{\overline{1}}}\,.
\]
We used the fact that $\{z_1> z_2\}=\{5>4\}$, so $\{t_1>t_2\}=\{2>1\}$. 
}
\end{example}

We will consider Definition~\ref{defkn} as the definition of KN columns. 
 
%%%%%%%%%%%%%%%%%%%%%%%%%%%%%%%%%%%%%%%%%%%%%%%%%%%
\subsection{Kirillov--Reshetikhin crystals}
For the definition of the crystal energy function, we need to endow the KN columns with an affine
crystal structure. These finite-dimensional affine crystals are called {\em Kirillov--Reshetikhin (KR) crystals}.
Combinatorial models for all non-exceptional types were provided in~\cite{FOS:2009}.

Here we only describe the KR crystals $B^{r,1}$ for types $A_{n-1}^{(1)}$ and $C_n^{(1)}$, where
$r\in \{1,2,\ldots,n-1\}$ and $r\in \{1,2,\ldots,n\}$, respectively. As a classical type $A_{n-1}$ (resp. $C_n$) crystal, 
the KR crystal is isomorphic to
\[
	B^{r,1} \cong B(\omega_r).
\]

The crystal operator $f_0$ is given as follows. Let $b\in B^{k,1}$,
represented by a one-column KN tableau. In type $A_{n-1}$, if $b$ contains the letter $n$ and no $1$,
$f_0(b)$ is the obtained from $b$ by removing $n$ and adding $1$ to the column, leaving all letters
in strictly increasing order. Otherwise $f_0(b) = \emptyset$. In type $C_n$, if $b$ contains the letter $\overline{1}$,
then $f_0(b)$ is obtained from $b$ by removing the $\overline{1}$ and adding the letter $1$, arranging
all letters again in strictly increasing order. Otherwise $f_0(b) = \emptyset$. Note that if $b$ contains 
$\overline{1}$, then it cannot contain $1$ by the KN condition of Definition~\ref{definition.KN_column}.

Similarly, in type $A_{n-1}$, $e_0(b)$ changes a $1$ into $n$ if $1$ is in $b$ and $n$ is not, and otherwise 
$e_0(b)= \emptyset$. In type $C_n$, $e_0(b)$ is obtained from $b$ by changing a $1$ into a $\overline{1}$
if it exists, and $e_0(b) = \emptyset$ otherwise.

%%%%%%%%%%%%%%%%%%%%%%%%%%%%%%%%%%%%%%%%%%%%%%%%%%%
\subsection{The $D$ function}
 
Now we come to the definition of the energy function $D$ on tensor products of KR crystals $B^{r,1}$
of type $A_{n-1}^{(1)}$ or $C_n^{(1)}$.
It is defined by summing up combinatorially defined ``local'' energy contributions using the combinatorial
$R$-matrix.

Let $B_1$, $B_2$ be two affine crystals with generators $v_1$ and $v_2$, respectively, such 
that $B_1\otimes B_2$ is connected and $v_1\otimes v_2$ lies in a one-dimensional weight space. 
By~\cite[Proposition 3.8]{LOS:2010}, this holds for any two 
KR crystals. The generator $v_{r,s}$ for the KR crystal $B^{r,s}$ is the unique element of classical weight $s\omega_r$.

The \textit{combinatorial $R$-matrix}~\cite[Section 4]{KMN1:1992} is the unique 
crystal isomorphism
\begin{equation*}
    \sigma : B_2 \otimes B_1 \to B_1 \otimes B_2.
\end{equation*}
By weight considerations, this must satisfy $\sigma(v_2 \otimes v_1) = v_1 \otimes v_2$.
 
As in~\cite{KMN1:1992} and~\cite[Theorem 2.4]{OSS:2003}, there is a function 
$H=H_{B_2,B_1}:B_2\otimes B_1\rightarrow\Z$, unique up to a global additive constant, such that, 
for all $b_2\in B_2$ and $b_1\in B_1$,
\begin{equation} \label{eq:local energy}
  H(e_i(b_2\otimes b_1))=
  H(b_2\otimes b_1)+
  \begin{cases}
    -1 & \text{if $i=0$ and LL,}\\
    1 & \text{if $i=0$ and RR,}\\
    0 & \text{otherwise.}
  \end{cases}
\end{equation}
Here LL (resp. RR) indicates that $e_0$ acts on the left (resp. right) tensor factor in both
$b_2\otimes b_1$ and $\sigma (b_2 \otimes b_1)$. When $B_1$ and $B_2$ are KR crystals, we normalize 
$H_{B_2, B_1}$ by requiring $H_{B_2, B_1}(v_2 \otimes v_1)= 0$, where $v_1$ and $v_2$ are the generators
defined above. The map $H$ is called the {\em local energy function}.
 
\begin{definition} \label{definition.D_B}
For $B=B^{r_N,1} \otimes \cdots \otimes B^{r_1,1}$ of type $A_{n-1}^{(1)}$ or $C_n^{(1)}$, set
\begin{equation*}
        H^R_{j,i}:= H_{i} \sigma_{i+1} \sigma_{i+2} \cdots \sigma_{j-1} \quad \text{and} \quad
        H^L_{j,i}:= H_{j-1} \sigma_{j-2} \sigma_{j-3} \cdots \sigma_i,
\end{equation*}
where $\sigma_j$ and $H_j$ act on the $j$-th and $(j+1)$-st tensor factors.
We define a right and left energy function $D^R_B, D_B^L : B \to \Z$ as
\begin{equation} \label{equation.D_B}
        D^R_B := \sum_{N \geq j > i \geq 1} H^R_{j,i} \quad \text{and} \quad
        D^L_B := \sum_{N \geq j > i \geq 1} H^L_{j,i}.
\end{equation}
We set $D_B:= D_B^L$ and, when there is no confusion, we shorten $D_B$ to simply $D$; this is referred to as the energy function. 
\end{definition}

\begin{remark}
Note that the energies $D_B^R$ and $D_B^L$ can be defined for general tensor products
of KR crystals. When the KR crystals decompose into several components as classical crystals
(unlike in the cases for type $A_{n-1}^{(1)}$ and $C_n^{(1)}$ we consider), there is an extra term in the
analogue of~\eqref{equation.D_B}; see~\cite{OSS:2003}.
\end{remark}

There is a precise relationship between $D^R$ and $D^L$ using the Lusztig involution.
To state it, let us introduce the following map
\begin{equation*}
\begin{split}
	\tau : \quad  B_N \otimes \cdots \otimes B_1 &\to B_1 \otimes \cdots \otimes B_N\\
	         b_N \otimes \cdots \otimes b_1 &\mapsto S(b_1) \otimes \cdots \otimes S(b_N).
\end{split}
\end{equation*}

For types $A_{n-1}^{(1)}$ and $C_n^{(1)}$ and $B_i=B^{r_i,1}$, the KR crystal $B_i$ is connected as a classical 
crystal and under $S$ the classically highest weight element $u_i^{\mathrm{highest}}$ maps to the lowest weight element
$u_i^{\mathrm{lowest}}$. It is not hard to show from the explicit description of $S$, $e_0$ and $f_0$ in this case, 
that the following diagrams commute:
\begin{equation} \label{equation.commutative_tau}
   \raisebox{0.7cm}{
   \xymatrix{
     B^{r,1} \ar[r]^{f_0} \ar[d]^S & B^{r,1} \ar[d]^S \\
     B^{r,1} \ar[r]_{e_0} & B^{r,1}
   }}
   \qquad \text{and} \qquad
   \raisebox{0.7cm}{
   \xymatrix{
     B \ar[r]^{f_0} \ar[d]^\tau & B \ar[d]^\tau \\
     \widetilde{B} \ar[r]_{e_0} & \widetilde{B}
   }}
\end{equation}
where $B= B^{r_N,1} \otimes \cdots \otimes B^{r_1,1}$ and $\widetilde{B} = B^{r_1,1} \otimes \cdots \otimes B^{r_N,1}$.

This shows in particular that the crystal commutor~\eqref{equation.commutor} is lifted to an
affine crystal isomorphism in these cases and hence must coincide with the combinatorial $R$-matrix $\sigma$.

\begin{proposition} \label{proposition.Dtau}
Let $B=B^{r_N,1} \otimes \cdots \otimes B^{r_1,1}$ of type $A_{n-1}^{(1)}$ or $C_n^{(1)}$
and $b\in B$. Then
\begin{equation} \label{equation.DS}
	D_B^R(b) = D_B^L(\tau(b)).
\end{equation}
\end{proposition}

\begin{proof}
Note that if $\sigma(b_1 \otimes b_2) = b_2' \otimes b_1'$, then $\sigma(S(b_2) \otimes S(b_1)) = 
S(b_1') \otimes S(b_2')$ since $\sigma$ and the commutor~\eqref{equation.commutor} agree on two tensor factors
under the conditions of the proposition by the above arguments.
Hence the terms in the sum of $D^R$ are in one-to-one correspondence with terms in $D^L$. Therefore it suffices to 
show that the local energy satisfies
\begin{equation} \label{equation.HS}
	H(b_2 \otimes b_1) = H(S(b_1) \otimes S(b_2))
\end{equation}
for $b_2 \otimes b_1 \in B^{r_2,1} \otimes B^{r_1,1}$.
Since $S(u_1^{\mathrm{highest}}) \otimes S(u_2^{\mathrm{highest}})
= u_1^{\mathrm{lowest}} \otimes u_2^{\mathrm{lowest}}$ lies in the same component as 
$u_1^{\mathrm{highest}} \otimes u_2^{\mathrm{highest}}$, we have
\[
	H(u_2^{\mathrm{highest}} \otimes u_1^{\mathrm{highest}}) 
	=H(S(u_1^{\mathrm{highest}}) \otimes S(u_2^{\mathrm{highest}})) = 0.
\]
In the recursion~\eqref{eq:local energy}, if for example $e_0$ acts LL on $b_2\otimes b_1$,
then $f_0$ acts RR on $S(b_1) \otimes S(b_2) = S(b_1) \otimes S(b_2)$ by~\eqref{equation.commutative_tau}.
Hence $H$ changes by the same amount in the left and right hand side of~\eqref{equation.HS}. 
The other cases can be checked analogously, which proves~\eqref{equation.HS}.
\end{proof}

%%%%%%%%%%%%%%%%%%%%%%%%%%%%%%%%%%%%%%%%%%%%%%%%%%% 
\subsection{$D$ energy as affine grading}

As suggested in~\cite[Section 2.5]{OSS:2002} and proven in~\cite{ST:2011},
the energy $D^R$ is the same as the affine degree grading in the associated highest weight affine crystals
up to an overall shift. We will explain this now since it plays a crucial role in the proof of the equality between
charge and energy. 

We begin with the definition of Demazure arrows. For this we need constants $c_r$ for $r\in I\setminus \{0\}$
as for example defined in~\cite{FOS:2010}. In the cases of concern to us here, we have $c_r=1$ for all $r$
in type $A_{n-1}^{(1)}$ and $c_r=2$ for $1\le r<n$ and $c_n=1$ in type $C_n^{(1)}$.

\begin{definition} \label{definition.demazure_arrows}
Let $B= B^{r_N, s_N} \otimes \cdots \otimes  B^{r_1, s_1}$ be a tensor product of KR crystals and fix an 
integer $\ell$ such that $\ell \geq \lceil s_k/c_k \rceil$ for all $1 \leq k \leq N$. We call such a tensor 
product a \textit{composite KR crystal of level bounded by $\ell$}.

An arrow $f_i$ is called an \textit{$\ell$-Demazure arrow} on $b\in B$ if $\varphi_i(b)>0$ and either 
$i\in I\setminus \{0\}$ or $i=0$ and $\varepsilon_0(b)\ge \ell$. 
\end{definition}

In the setting of this paper, we are only concerned with tensor products of types $A_{n-1}^{(1)}$ and 
$C_n^{(1)}$ of the form $B= B^{r_N,1} \otimes \cdots \otimes B^{r_1,1}$. In this case one can pick $\ell=1$
and a \textit{Demazure arrow} for $B$ is a $1$-Demazure arrow.

\begin{lemma}\label{rec-const-energy}
Let $B= B^{r_N, 1} \otimes \cdots \otimes  B^{r_1, 1}$ of type $A_{n-1}^{(1)}$ or $C_n^{(1)}$ and $b\in B$. Then
\begin{enumerate}
\item \label{i.R}$\varepsilon_0(b) \ge 1$ implies $D^R(f_0(b)) = D^R(b)+1$;
\item \label{i.L} $\varphi_0(b) \ge 1$ implies $D^L(e_0(b)) = D^L(b)+1$.
\end{enumerate}
\end{lemma}

\begin{proof}
Part~\eqref{i.R} follows directly from~\cite[Lemma 7.3]{ST:2011}. For part~\eqref{i.L}, recall that 
by~\eqref{equation.commutative_tau} and Proposition~\ref{proposition.Dtau} we have
$e_0(\tau(b)) = \tau(f_0(b))$ and $D^L(\tau(b)) = D^R(b)$. Also, setting $\tilde{b}=\tau(b)$, we have
$\varphi_0(\tilde{b})\ge 1$ if $\varepsilon_0(b)\ge 1$; thus, by using part~\eqref{i.R}, we deduce
\begin{equation*}
	D^L(e_0(\tilde{b})) = D^L(e_0(\tau(b))) = D^L(\tau(f_0(b))) = D^R(f_0(b))
	= D^R(b)+1 = D^L(\tilde{b})+1\,,
\end{equation*}
which proves the claim.
\end{proof}

The proof of the following essentially appeared in~\cite[Proof of Theorem 4.4.1]{KMN1:1992} and 
was spelled out in this precise form in~\cite[Proposition 8.1]{ST:2011}. Here 
\begin{equation*}
	P_\ell^+ = \{ \lambda \in P^+ \mid \lev(\lambda) = \ell \},
\end{equation*}
where $\lev(\lambda):=\lambda(c)$ is the level of $\lambda$ and $c$ is the central element
$c=\sum_{i\in I} a_i^\vee \alpha_i^\vee$.

\begin{proposition} \label{proposition.expansion_nonperfect}
For $B$ a composite KR crystal of level bounded by $\ell$,
\begin{equation} \label{equation.expansion_nonperfect}
	B \otimes B(\ell \Lambda_0) \cong \bigoplus_{\Lambda'} B(\Lambda'),
\end{equation}
where the sum is over a finite collection of (not necessarily distinct) $\Lambda' \in P^+_\ell$. 
\end{proposition}

In Section~\ref{section.kyotoC}, we discuss in more detail the sum on the right hand side of~\eqref{equation.expansion_nonperfect} for type $C_n^{(1)}$ and $\ell=1$.

\begin{definition} 
For each $b \in B$, let $u_b^{\ell \Lambda_0}$ be the unique element of $B$ 
such that $u^{\ell\Lambda_0}_b \otimes u_{\ell \Lambda_0}$ is the highest weight in the component from 
Proposition~{\rm \ref{proposition.expansion_nonperfect}} containing $b \otimes u_{\ell \Lambda_0}$. 
\end{definition}

Define the function $\deg$ on a direct sum of highest weight crystals to be the basic grading on each component, 
with all highest weight elements placed in degree $0$. 

\begin{corollary} \label{corollary.D=degree}
Choose an isomorphism $m: B \otimes B( \ell \Lambda_0) \cong \bigoplus_{\Lambda'} B(\Lambda')$. 
Then for all $b \in B$, we have $D(b)-D(u_b^{\ell \Lambda_0})= \deg(m(b \otimes u_{ \ell \Lambda_0}))$.
\end{corollary}

\begin{corollary} \label{corollary.D=degree1}
The minimal number of $e_0$ in a string of $e_i$ taking $b$ to $u_b^{\ell \Lambda_0}$ is 
$D(b)-D(u_b^{\ell \Lambda_0})$. 
\end{corollary}

\begin{remark}
One special case of interest is when $B$ is a tensor product of perfect KR crystals of level $\ell$, 
and $\Lambda$ also has level exactly $\ell$. Then the right side of~\eqref{equation.expansion_nonperfect}
is a single highest weight crystal and the isomorphism in Proposition~\ref{proposition.expansion_nonperfect}
is used in the Kyoto path model~\cite{KMN1:1992}. Hence $u_b^{\ell \Lambda_0}$ does not in fact depend 
on $b$, simplifying Corollaries~\ref{corollary.D=degree} and~\ref{corollary.D=degree1}.
\end{remark}

%%%%%%%%%%%%%%%%%%%%%%%%%%%%%%%%%%%%%%%%%%%%%%%%%%% 
\section{The charge construction}
\label{section.chargerev}

%%%%%%%%%%%%%%%%%%%%%%%%%%%%%%%%%%%%%%%%%%%%%%%%%%% 
\subsection{The classical charge}\label{constra} 
Let us start by recalling the construction of the classical charge of a word due to Lascoux and 
Sch\"utzenberger~\cite{LS:1979}. Assume that $w$ is a word with letters in the alphabet 
$[n]:=\{1,\ldots,n\}$ which has partition content, i.e., the number of $j$'s is greater than or equal to
the number of $j+1$'s, for each $j=1,\ldots,n-1$. The statistic $\charge(w)$ is calculated as a sum 
based on the following algorithm. Scan the  word starting from its right end, and select the numbers 
$1,2,\ldots$ in this order, up to the largest possible $k$. We always pick the first available entry $j+1$ 
to the left of the previous entry $j$. Whenever there is no such entry, we pick the rightmost entry $j+1$, 
so we start scanning the word from its right end once again; in this case, we also add $k-j$ to the sum 
that computes $\charge(w)$. At the end of this process, we remove the selected numbers and repeat 
the whole procedure until the word becomes empty. 

\begin{example}\label{lsex}
{\rm 
Consider the word $w=11{\mathbf 3}2\mathbf{214}323$, where the first group of selected numbers 
is shown in bold. The corresponding contribution to the charge is $1$. After removing the bold numbers 
and another round of selections (again shown in bold), we have $1{\mathbf 1}2\mathbf{32}3$, so the 
contribution to the charge is $2$. We are left with the word $123$, whose contribution to the charge is 
$2+1=3$. So $\charge(w)=1+2+(2+1)=6$.
}
\end{example}
 
We now reinterpret the classical charge as a statistic on a tensor product of type $A_{n-1}^{(1)}$ 
KR crystals. Such a crystal indexed by a column 
of height $k$ is traditionally denoted $B^{k,1}$, and its vertices are indexed by increasing fillings of 
the mentioned column with integers in $[n]$. Given a partition $\mu$ (i.e., a dominant weight in the root 
system), let 
\begin{equation} \label{equation.tensor_mu}
	B_\mu:=\bigotimes_{i=1}^{\mu_1} B^{\mu_i',1}\,,
\end{equation}
where $\mu'$ is the conjugate partition to $\mu$.
This is simply the set of column-strict fillings of the Young diagram $\mu$ with integers in $[n]$. 
Note that unlike in Section~\ref{section.energy}, the tensor factors in~\eqref{equation.tensor_mu}
are ordered in weakly decreasing order.

Fix a filling $b$ in $B_\mu$ written as a concatenation of columns $b_1\ldots b_{\mu_1}$. We 
attach to it a filling $c:=\mbox{circ-ord}(b)=c_1\ldots c_{\mu_1}$ according to the following algorithm, which is based 
on the circular order $\prec_i$ on $[n]$ starting at $i$, namely 
$i\prec_i i+1\prec_i\cdots \prec_i n\prec_i 1\prec_i\cdots\prec_i i-1$. 

\begin{algorithm}\label{algreord}\hfill\\
let $c_1:=b_1$;\\
for $j$ from $2$ to $\mu_1$ do\\
\indent for $i$ from $1$ to $\mu_j'$ do\\
\indent\indent let $c_j(i):=\min\,(b_j\setminus\{c_j(1),\ldots,c_j(i-1)\},\,\prec_{c_{j-1}(i)})$\\
\indent end do;\\
end do;\\
return $c:=c_1\ldots c_{\mu_1}$.\\
\end{algorithm}

\begin{example}\label{exrinv}
{\rm 
Algorithm~\ref{algreord} constructs the filling $c$ from the filling $b$ below. The bold 
entries in $c$ are only relevant in Example {\rm \ref{exch}} below.
\begin{equation}\label{st}
	b=\tableau{{3}&{2}&{1}&{2}\\{5}&{3}&{2}\\{6}&{4}&{4}} \quad \text{and} \quad 
	c=\tableau{{3}&{3}&{\mathbf{4}}&{2}\\{\mathbf{5}}&{2}&{2}\\{\mathbf{6}}&{\mathbf 4}&{1}}\,.
\end{equation}
}
\end{example}

We introduce some terminology in order to reinterpret the classical charge in terms of a statistic 
on $B_\mu$. Given the considered filling $b$ in $B_\mu$, we define its {\em charge word} as 
the biword $\cw(b)$ containing a biletter $\binom{k}{j}$ for each entry $k$ in the column $b_j$ 
of $b$. We order the biletters in the decreasing order of the $k$'s, and for equal $k$'s, in the 
decreasing order of $j$'s. The obtained word formed by the lower letters $j$ will be denoted by 
$\cw_2(b)$. We refer to Example \ref{exch} for an illustration of the charge word. On the other
hand, given the filling $c=c_{1}\ldots c_{\mu_1}$ constructed by Algorithm \ref{algreord}, 
we say that the cell $\gamma$ in column $c_j$ and row $i$ is a descent if $c_j(i)>c_{j+1}(i)$, assuming that 
$c_{j+1}(i)$ is defined. Let ${\rm Des}(c)$ denote the set of descents in $c$. As usual, 
we define the arm length ${\rm arm}(\gamma)$ of a cell $\gamma$ as the number of cells to its right.

It is not hard to see that Algorithm~\ref{algreord} for constructing $c$ from $b$ translates 
precisely into the selection algorithm which computes $\charge(\cw_2(b))$. More precisely, 
consider the $i$th sequence $1,2,\ldots$ extracted from $\cw_2(b)$ (which turns out to have 
length $\mu_i$), and the letter $j$ in this sequence; then the top letter paired with the mentioned 
letter $j$ in $\cw(b)$ is precisely the entry $c_j(i)$ in row $i$ and column $j$ of the filling 
$c$. In particular, the steps to the right in the $i$th iteration of the charge computation 
correspond precisely to the descents in the $i$th row of $c$, while the corresponding 
charge contributions and arm lengths coincide. We conclude that 
\begin{equation}\label{chdesa}
	\sum_{\gamma\in{\rm Des}(c)}{\rm arm}(\gamma)=\charge(\cw_2(b))\,.
\end{equation}
For simplicity, we set $\charge(b):=\charge(\cw_2(b))$.

\begin{remark}
In~\cite{lenmpc} we showed that the charge statistic on $B_\mu$ can be derived from the 
Ram--Yip formula~\cite{raycfm} for the corresponding Macdonald polynomial at $t=0$. In fact, 
we showed that Algorithm~\ref{algreord} is closely related to the corresponding quantum Bruhat 
graph (see, e.g., \cite{bfpmbo}). So we can conclude that this graph explains the charge construction
itself. The mentioned idea was extended to type $C$, and it led to the definition of a type $C$ 
charge, that we describe in Section~\ref{constrc}.
\end{remark}

\begin{example}\label{exch}
{\rm 
Note that $\cw_2(b)$ for $b$ in Example~\ref{exrinv} is precisely the word $w$ in 
Example~\ref{lsex}. In fact, the full biword $\cw(b)$ is shown below, using the order on the 
biletters specified above. The index attached to a lower letter is the number of the iteration in 
which the given letter is selected in the process of computing $\charge(b)$:
\[
	\cw(b)=\left(\begin{array}{cccccccccc}6&5&4&4&3&3&2&2&2&1\\
	1_3&1_2&3_1&2_3&2_1&1_1&4_1&3_2&2_2&3_3\end{array}\right)\,.
\]
One can note the parallel between the mentioned selection process and the construction 
of $c$ from $b$ in Example~\ref{exrinv}. The entries in the cells of ${\rm Des}(c)$ 
are shown in bold in~\eqref{st}.
}
\end{example}

%%%%%%%%%%%%%%%%%%%%%%%%%%%%%%%%%%%%%%%%%%%%%%%%%%%%%%%%%%%%%%%
\subsection{The type $C$ charge}\label{constrc} 
In this section we recall from \cite{lenmpc} the construction of the type $C$ charge. We start by fixing a 
dominant weight $\mu$ in the root system of type $C_n$. Let 
\begin{equation}\label{bmuc}
	B_\mu:=\bigotimes_{i=1}^{\mu_1} B^{\mu_i',1}\,,
\end{equation}
where $B^{k,1}$ is the type $C_n^{(1)}$ KR crystal indexed by a column of height 
$k$. Note that $B_\mu$ is the set of fillings $b=b_1\ldots b_{\mu_1}$ of the shape $\mu$ with integers in $[\overline{n}]$ whose 
columns $b_j$ are KN columns; indeed, the KN columns of height $k$ label the vertices of $B^{k,1}$. As 
mentioned above, it will be more useful to represent $b_j$ in the split form $b^L_j b^R_j$; in this case, $b$ becomes 
a filling $b^L_{1}b^R_{1}\ldots b^L_{\mu_1}b^R_{\mu_1}$ of the shape $2\mu$.   

Now fix a filling $b$ in $B_\mu$, represented with split columns, which are labeled from left to right 
by $1,1',2,2',\ldots$. We can apply a slight modification of Algorithm \ref{algreord} to $b$ and 
obtain a filling $c=c^L_{1}c^R_{1}\ldots c^L_{\mu_1}c^R_{\mu_1}=\mbox{circ-ord}(b)$ of $2\mu$; namely, we start 
by setting $c^L_1:=b^L_1$, and then consider the (doubled) columns of $b$ from left to right. We use the 
circular order on $[\overline{n}]$ starting at various values $i$, which we still denote by $\prec_i$. 

\begin{example}\label{exchc}
{\rm 
Consider the following tensor product of KN columns:
\[
	\tableau{{\overline{5}}\\{\overline{3}}\\{\overline{2}}\\{\overline{1}}}\otimes \tableau{{3}\\
	{\overline{4}}\\{\overline{3}}}\otimes\tableau{{1}\\{3}\\{\overline{3}}} \,.
\]
This is represented with split columns as the following filling $b$ of the shape $2\mu=(6,6,6,2)$:
\[
	\tableau{{1}&{1'}&{2}&{2'}&{3}&{3'}\\ \\{\overline{5}}&{\overline{5}}&{2}&{3}&{1}&{1}\\{\overline{3}}&
	{\overline{3}}&{\overline{4}}&{\overline{4}}&{2}&{3}\\
	{\overline{2}}&{\overline{2}}&{\overline{3}}&{\overline{2}}&{\overline{3}}&{\overline{2}}\\
	{\overline{1}}&{\overline{1}}}\,,
\]
where the top row consists of the column labels. The corresponding filling $c$ is
\[
	\tableau{{\overline{5}}&{\overline{5}}&{\overline{4}}&{\overline{4}}&{\overline{3}}&
	{\overline{2}}\\{\overline{3}}&{\overline{3}}&{\overline{3}}&{\mathbf{\overline{2}}}&{1}&{1}\\
	{\overline{2}}&{\mathbf{\overline{2}}}&{{2}}&{\mathbf{3}}&{{2}}&   {{3}}\\
	{\overline{1}}&{\overline{1}}}\,.
\] 
}
\end{example}

Define the charge word $\cw(b)$ of $b$ by analogy with type $A$, as the biword containing 
a biletter $\binom{k}{j}$ for each entry $k$ in column $j$ of $b$; here $j$ and $k$ belong to 
the alphabets $\{1<1'<2<2'<\ldots\}$ and $[\overline{n}]$, respectively. We order the biletters as 
in the type $A$ case (in the decreasing order of the $k$'s, and for equal $k$'s, in the decreasing 
order of $j$'s), and define $\cw_2(b)$ in the same way (as the word formed by the lower letters 
$j$). 

The modification of Algorithm~\ref{algreord} for constructing $c$ from $b$ can be rephrased 
in terms of $\cw(b)$, as explained below; we will refer to this rephrasing as the charge algorithm. 
We start by scanning $\cw_2(b)$ from right to left and by selecting the entries 
$1,1',2,2',\ldots,\mu_1,(\mu_1)'$ in this order, according to the following rule: always pick the first 
available entry to the left, but if the desired entry is not available then scan the word from its right 
end once again. As in type $A$, we can see that the sequence of top letters paired with 
$1,1',2,2',\ldots,\mu_1,(\mu_1)'$ is the first row of the filling $c$ (read from right to left). We 
then remove the selected entries from $\cw_2(b)$ and repeat the above procedure, which will 
now give the other rows of $c$, from top to bottom. It was shown in \cite{lenmpc} that we 
always go left from $j$ to $j'$, but we can go right from $j'$ to $j+1$.

\begin{example}
{\rm 
This is a continuation of Example~\ref{exchc}. The charge word $\cw(b)$, with the order on the 
biletters indicated above, is
\[
	\left(\begin{array}{cccccccccccccccccccc}\overline{1}&\overline{1}&\overline{2}&\overline{2}&
	\overline{2}&\overline{2}&\overline{3}&\overline{3}&\overline{3}&\overline{3}&\overline{4}&
	\overline{4}&\overline{5}&\overline{5}&3&3&2&2&1&1\\
	1'_4&1_4& 3'_1&2'_2&1'_3&1_3&  3_1&2_2&1'_2&1_2&  2'_1&2_1&  1'_1&1_1&  3'_3&2'_3&  
	3_3&2_3&  3'_2&3_2
	\end{array}\right)   \,.
\] 
The index attached to a lower letter is the number of the iteration in which the given letter is
selected by the charge algorithm.
}
\end{example}

Descents are defined as usual, cf. Section~\ref{constra}. It is easy to see that the descents in $c$ 
correspond to the steps to the right in the charge algorithm applied to $\cw_2(b)$. By an observation 
made above, we only have descents of the form $c^R_j(i)>c^L_{j+1}(i)$. We are led to the following 
definition of the type $C$ charge.

\begin{definition} 
Consider a word $w$ with letters in the alphabet $1,1',2,2',\ldots$, containing as many letters $j$ 
as $j'$, and at least as many letters $j$ as $j+1$. Apply the charge algorithm to $w$, and assume 
that a selected entry $j'$ is always to the left of the previously selected $j$. Let $\charge(w)$ be 
the sum of $k-j$ for each selected entry $j+1$ to the right of the previously selected $j'$, where 
the selected entries in the given iteration are $1,1',\ldots,k,k'$. 
\end{definition}

The above discussion leads to the following result:
\begin{equation}\label{chdesc}
	\frac{1}{2}\sum_{\gamma\in{\rm Des}(c)}{\rm arm}(\gamma)=\charge(\cw_2(b))\,.
\end{equation}
For simplicity, we again set $\charge(b):=\charge(\cw_2(b))$.

\begin{example}
{\rm 
This is still a continuation of Example~\ref{exchc}. The entries in the descents of $c$ are 
shown above in bold. Correspondingly, the charge algorithm applied to $\cw_2(b)$ makes 
one step to the right in the second iteration (from $2'$ to $3$), and two steps to the right in the 
third iteration (from $1'$ to $2$ and from $2'$ to $3$). Thus, $\charge(b)=1+(2+1)=4$. 
}
\end{example}

%%%%%%%%%%%%%%%%%%%%%%%%%%%%%%%%%%%%%%%%%%%%%%%%%%%%%%%%%%%%%%%
\section{Energy and charge in type $A$}
\label{section.energyA}
%%%%%%%%%%%%%%%%%%%%%%%%%%%%%%%%%%%%%%%%%%%%%%%%%%%%%%%%%%%%%%%

In this section we rederive the result of Nakayashiki--Yamada~\cite{NY:1997} showing the equality of
the energy function and charge in type $A_{n-1}$. We do this in a more conceptual way, by using the method of Demazure arrows
(the proof in \cite{NY:1997} is based on subtle combinatorics of Young tableaux). Furthermore, we work 
with all the crystal vertices in a tensor product of columns, not just the highest weight vertices considered in 
\cite{NY:1997}. Note that the setup in the mentioned paper is that of the right energy and a tensor product of columns of increasing heights which, by Proposition \ref{proposition.Dtau}, is equivalent to the setup in this paper.

We start by studying the behavior of the type $A$ charge with respect to the crystal operators.

\begin{proposition} \label{chaf1n}
The type $A_{n-1}$ charge is preserved by the crystal operators $f_1,\ldots,f_{n-1}$.
\end{proposition}

\begin{proof}
Let $b$ be a tensor product of columns in some $B_\mu$ (see Section \ref{constra} and the terminology 
therein, which we use freely).  It is known that the word $\cw_2(f_i(b))$ is in the same plactic equivalence 
class (see, e.g., \cite{lltpm}) as $\cw_2(b)$. More precisely, the former is obtained from the latter by considering 
its subword formed by the letters $x$ corresponding to biletters $\left(\!\!\begin{array}{c}i+1\\x\end{array}\!\!\right)$ 
or $\left(\!\!\begin{array}{c}i\\x\end{array}\!\!\right)$ in $\cw(b)$, by viewing this subword as (the column word of) 
a skew tableau with two columns, and by using jeu de taquin to move a letter from the right column to the left 
one. This is explained in detail in \cite[Section 2]{lasdcg}, based on the notion of ``double crystal graphs''. Then 
we use the well-known fact that the classical charge is preserved by the plactic relations (see, 
e.g., \cite[Lemma 6.6.6 (ii)]{lltpm}). 
\end{proof}

\begin{proposition} \label{chae0}
Let $B=B^{r_N,1}\otimes \cdots \otimes B^{r_1,1}$ be of type $A_{n-1}^{(1)}$ with $r_N\ge r_{N-1} \ge \cdots \ge
r_1$ and $b\in B$. If $\varphi_0(b)\ge 1$ and $\varepsilon_0(b)\ge 1$, then the type $A_{n-1}$ 
charge satisfies
\[
	\charge(e_0(b))=\charge(b)-1\,.
\]
\end{proposition}

\begin{proof}
Let $b=b_1\ldots b_{\mu_1}$, and assume that $e_0$ changes the entry $1$ in $b_j$ to $n$. The condition 
$\varphi_0(b)\ge 1$ implies $j>1$. Let $c:=\mbox{circ-ord}(b)$ with $c=c_1\ldots c_{\mu_1}$, where $c_j(i)=1$, and 
$d:=\mbox{circ-ord}(e_0(b))$ with $d=d_1\ldots d_{\mu_1}$ (recall that the map $\mbox{circ-ord}$ is defined by 
Algorithm~\ref{algreord}). 

{\em The main case.} We start by assuming that $c_{j-1}(i)>1$ and that the column $c_{j+1}$, if it exists, does 
not contain $n$ in a row $k\ge i$. We claim that, in this case, $d_l(k)=c_l(k)$ for all $k,l$ with the exception of 
$d_j(i)=n$. Indeed, all we need to check are the following, in this order: (1) the entries $d_j(k)$ for 
$k=1,\ldots,i-1$ are as claimed, since the alternative, namely that one of them is $n$, would lead to a 
contradiction; (2)  the value of $d_j(i)$ follows from the fact that $c_{j-1}(i)\ne 1$; (3) the value of 
$d_{j+1}(i)$, if this entry exists, follows from the above condition on the column $c_{j+1}$. By the above 
facts, we have the same descents in the fillings $c$ and $d$, with the exception of the descent 
$c_{j-1}(i)>c_j(i)=1$, which corresponds to $d_j(i)=n>d_{j+1}(i)$, assuming that $d_{j+1}(i)$ exists. 
The difference between the arm lengths of the mentioned descents is $1$, which concludes the 
proof by (\ref{chdesa}). The remaining part of the proof treats the exceptions to this case. 

{\em Exception {\rm 1}.} Assume that $c_{j-1}(i)=1$. Then the column $b_{j-1}$ must contain $n$ (otherwise 
the entry $1$ in $b_j$ would not be the leftmost unpaired $1$). Assuming that $c_{j-1}(k)=n$, we must have 
$k>i$, by Algorithm \ref{algreord} and the fact that $c_j(i)=1$. The condition  $\varphi_0(b)\ge 1$ implies 
$j>2$, and we have $c_{j-2}(i)=1$ (otherwise $c_{j-1}(i)=1$ and $c_{j-1}(k)=n$ for $k>i$ are in contradiction 
with the way in which Algorithm \ref{algreord} reorders the column $b_{j-1}$). This reasoning can be continued
indefinitely, so this case is impossible. 

{\em Exception {\rm 2}.} Assume that $c_{j-1}(i)>1$, but $c_{j+1}(k)=n$ for some $k\ge i$. Then $b_{j+1}$ 
must contain $1$ as well, namely $c_{j+1}(l)=1$ (otherwise the entry $1$ in $b_j$ would be paired with the 
entry $n$ in $b_{j+1}$). We must have $l\le i$, by Algorithm \ref{algreord} and the fact that $c_{j}(i)=1$. The 
above facts imply that $l<k$. 

{\em Exception {\rm 2.1}.} Assume that $l<i$. The fact that $c_{j}(l)>1$, $c_{j+1}(l)=1$, and $c_{j+1}(k)=n$ 
for $k>l$ are in contradiction with the way in which Algorithm \ref{algreord} reorders the column $b_{j+1}$. 
So this case is impossible.

{\em Exception {\rm 2.2}.} The only possibility left is that $c_{j+1}(i)=1$ and $c_{j+1}(k)=n$ for $k=k_1>i$. 
Let us assume for the moment that the second condition in Exception 2 holds for the column $c_{j+2}$, 
namely $c_{j+2}(k_2)=n$ for some $k_2\ge i$. By the same reasoning as above, we deduce that $c_{j+2}(i)=1$, 
so $k_2>i$. In fact, we also have $k_1\ge k_2$, by Algorithm \ref{algreord}. By continuing this reasoning, we obtain
\[c_{j+1}(i)=\ldots=c_{j+p}(i)=1\,,\;\;\;\;\;c_{j+1}(k_1)=\ldots=c_{j+p}(k_p)=n\,,\;\;\;\mbox{where $k_1\ge\ldots\ge k_p$}\,;\]
on the other hand, we can assume that, if the column $c_{j+p+1}$ exists, then it does not contain $n$ in rows 
$i,i+1,\ldots$. This information about the filling $c$ is represented in the figure below. The column $c_j$ is 
the leftmost column with an entry $1$ displayed, while $+$ and $*$ stand for entries different from $1$ and 
$n$, respectively. The boxes shown in bold represent descents. By a reasoning similar to the main case 
above, we deduce that the only difference between the fillings $c$ and $d$ consists of the entries $1$ 
and $n$ in the figure changing to $n$ and $1$ in $d$, respectively. This leads to the marked descents 
moving to the boxes indicated by the arrows. It is now easy to see that the sum of the arm lengths of 
descents decreases by $1$ when passing from $c$ to $d$.

\pagebreak

$\;$

\vspace{-17cm}

\[\;\;\;\;\;\;\;\;\;\;\;\;\;\;\mbox{\includegraphics[scale=0.75]{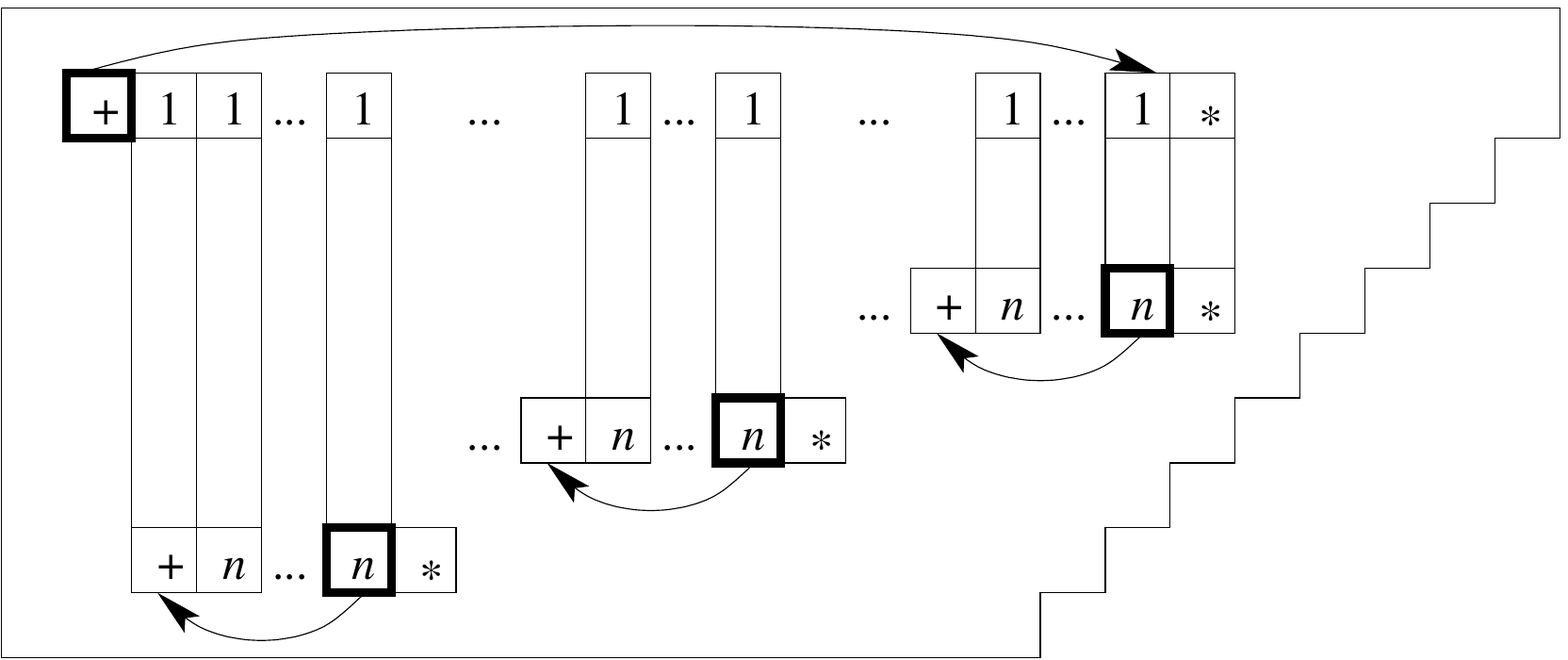}}\]
\end{proof}

\begin{proof}[Proof of Theorem {\rm \ref{theorem.main}} in type $A$]
The proof is immediate based on Corollary~\ref{corollary.D=degree} and Propositions~\ref{chaf1n}, \ref{chae0}
using the fact that KR crystals of type $A_{n-1}^{(1)}$ are perfect.
\end{proof}

%%%%%%%%%%%%%%%%%%%%%%%%%%%%%%%%%%%%%%%%%%%%%%%%%%% 
\section{Kyoto path model for nonperfect type $C$}
\label{section.kyotoC}
%%%%%%%%%%%%%%%%%%%%%%%%%%%%%%%%%%%%%%%%%%%%%%%%%%% 

In this section, we make Proposition~\ref{proposition.expansion_nonperfect} more explicit in the case
of $B=B^{r_N,1} \otimes \cdots \otimes B^{r_1,1}$ and $\ell=1$ for type $C_n^{(1)}$, 
by providing a correspondence between highest weight elements (or ground states)
in $B\otimes B(\Lambda_0)$ and elements in $B(\Lambda')$ in the sum on the right hand side 
of~\eqref{equation.expansion_nonperfect}, which are of type $A$. This will help in the next section
to prove Theorem~\ref{theorem.main} for type $C_n^{(1)}$.

We call the highest weight elements in $B\otimes B(\Lambda_0)$ \textit{ground state paths}.
There is a recursive construction for them, which starts by listing all elements $b_1 \in B^{r_1,1}$ such that
$\varepsilon(b_1) = \Lambda_0$. Suppose $b_k \otimes \cdots \otimes b_1 \in
B^{r_k,1} \otimes \cdots \otimes B^{r_1,1}$ are already constructed. Then $b_{k+1} \in B^{r_{k+1},1}$
can be any of the elements such that $\varepsilon(b_{k+1}) = \varphi(b_k)$. The weight of the
ground state is $\varphi(b_N)$, which is some fundamental weight $\Lambda_h$.
For perfect crystals there are unique elements $b_N, \ldots, b_1$ with the described properties.
However, in type $C_n^{(1)}$ the crystals $B^{r,1}$ are not perfect and the above construction
gives a tree of ground state elements.

\begin{example} \label{example.ground_states}
{\rm 
Take $B=B^{1,1} \otimes B^{2,1} \otimes B^{2,1} \otimes B^{3,1}$ of type $C_3^{(1)}$. Then
$b_1$ is the column $321$ and $b_2$ the column $\overline{2}\overline{3}$. For $b_3$
there are two choices, namely the columns $32$ or $\overline{2}2$. In the first case $b_4$
is $\overline{3}$, and in the second case $b_4$ can be $2$ or $\overline{1}$. In summary
the three ground states are
\[
	\tableau{\overline{3}} \otimes \tableau{2\\ 3}\otimes \tableau{{\overline{3}}\\{\overline{2}}}  \otimes \tableau{1\\ 2\\ 3} \otimes u_{\Lambda_0}
	\qquad
	\tableau{2} \otimes \tableau{2\\ {\overline{2}}}\otimes \tableau{{\overline{3}}\\{\overline{2}}}  \otimes \tableau{1\\ 2\\ 3} \otimes u_{\Lambda_0}
	\qquad 
	\tableau{\overline{1}} \otimes \tableau{2\\ {\overline{2}}}\otimes \tableau{{\overline{3}}\\{\overline{2}}}  \otimes \tableau{1\\ 2\\ 3} \otimes u_{\Lambda_0}\; .
\]
The weights are $\Lambda_2$, $\Lambda_2$, and $\Lambda_0$, respectively.
}
\end{example}

\begin{theorem}\label{theorem.ground2typeA}
Let $B=B^{r_N,1} \otimes \cdots \otimes B^{r_1,1}$ of type $C_n^{(1)}$.
From each ground state $u\otimes u_{\Lambda_0} \in B \otimes B(\Lambda_0)$ there exists a sequence 
of Demazure arrows $f_i$ (see Definition~{\rm \ref{definition.demazure_arrows}}), which ends at an element
$b\otimes u_{\Lambda_0}$ such that $b$ does not contain any barred letter.
\end{theorem}

Note that in Theorem~\ref{theorem.ground2typeA} there is no assumption on the order of the $r_i$.

In order to provide a proof of Theorem~\ref{theorem.ground2typeA}, we describe the explicit
sequence of $f_i$ satisfying the required conditions. For this we recursively define the following
objects.
\begin{itemize}
\item[-]
Let $\lambda_0 \subseteq \lambda_1 \subseteq \cdots \subseteq \lambda_s$ ($s$ is given below) be a sequence of 
shapes, where $\lambda_0$ is a single column of height $h$ if the weight of the ground state path
$u \otimes u_{\Lambda_0}$ is $\Lambda_h$. The other shapes are all of the form $\lambda(k,h_2,h_1)$, a 
partition with $k$ columns of height $n$ followed by two columns of heights $n>h_2\ge h_1\ge 0$, respectively.
If $\lambda_j = \lambda(k,h_2,h_1)$, then $\lambda_{j+1} = \lambda(k,h_2+1,h_1+1)$, where we identify
$\lambda(k,h_2+1,h_1+1)\cong \lambda(k+1,h_1+1,0)$ if $h_2+1=n$ and $h_1+1<n$, and 
$\lambda(k+2,0,0)$ if $h_2+1=h_1+1=n$. This adds one horizontal domino in the consecutive rightmost 
columns up to height $n$.
\item[-] We recursively define
\[
	v_{j+1} := F_j(v_j)
\] 
with $v_0=u$ the ground state and $F_j$ the sequence
\[
	F_j := f_0^{k+1} f_1^{2k+2} \cdots f_{h_1}^{2k+2} f_{h_1+1}^{2k+1} \cdots f_{h_2}^{2k+1} f_{h_2+1}^{2k} \cdots 
	f_{n-1}^{2k} f_n^k\,,
\]
if $\lambda_j=\lambda(k,h_2,h_1)$. We continue doing this as long as possible, in other words, until $F_s(v_s)$ is undefined.
\end{itemize}

\begin{example} \label{example.paths}
{\rm
Let $u \otimes u_{\Lambda_0}$ be the second ground state from Example~\ref{example.ground_states}. 
Then the sequence of $\lambda_0\subseteq \lambda_1 \subseteq \cdots \subseteq \lambda_s$ is
\[
	\tableau{\mbox{}\\ \mbox{}} \subseteq \tableau{\mbox{}&\mbox{}\\ \mbox{} \\ \mbox{}} \subseteq
	\tableau{\mbox{}&\mbox{}&\mbox{}\\ \mbox{}&\mbox{}\\ \mbox{}} \subseteq
	\tableau{\mbox{}&\mbox{}&\mbox{}\\ \mbox{}&\mbox{}& \mbox{} \\ \mbox{}&\mbox{}}
\]
with
\begin{equation*}
\begin{split}
	F_0 &= f_0 f_1 f_2\\
	F_1 &= f_0^2 f_1^3 f_2^2 f_3\\
	F_2 &= f_0^2 f_1^4 f_2^3 f_3.
\end{split}
\end{equation*}
}
\end{example}

Let $\lambda_j=(\lambda_{j1}\ge\lambda_{j2}\ge\ldots)$, and consider the conjugate partition 
$\lambda_j'=(\lambda_{j1}'\ge\lambda_{j2}'\ge\ldots)$. We will now explain how to represent an 
element $v_j$ by a collection of non-crossing lattice paths (which might touch each other); see 
Figure~\ref{figure.paths}. The paths have the following types of steps: down (southeast), up (northeast), 
and horizontal (east). More precisely, the paths in the collection $P_j=\{p_1,\ldots,p_l\}$ representing $v_j$
correspond to the $l=\lambda_{j1}$ columns of $\lambda_j$, and they satisfy the conditions below.
\begin{enumerate}
\item The endpoints of $p_1,\ldots,p_l$ are aligned at height 0 from right to left, and $p_i$ starts at 
height $\lambda_{ji}'$.
\item  The paths $p_2,\ldots,p_l$ and the segment of $p_1$ before the end of $p_2$ (in case $p_2$ exists) 
consist entirely of down and horizontal steps. Two up or down steps never lie below one another, and 
neither do only horizontal steps.
\item The path $p_i$ starts after $p_{i+2}$ ends, for $i=1,\ldots,l-2$.
\end{enumerate}
The terms ``before'' and ``after'' refer here to the order of the corresponding $x$-coordinates, with equality
allowed. Each down (resp. up) step starting at height $i$ is labeled by a letter $i$ (resp. $\bar{\imath}$). We
define ${\rm word}(P_j)$ as the word obtained by reading the labels on the paths in $P_j$ from left to right
(recall that two labels never lie below one another). 

Let us now explain the construction of the collections $P_j$. We start by defining $P_0$ as consisting of 
a single path: the one having the same word as  $u=v_0$, ending at height $0$, and having no horizontal 
steps; the word of $u$, denoted ${\rm word}(u)$, is defined as usual, by reading its columns from left to right, 
bottom to top. We construct the collections $P_j$ recursively, via the following transformation rule 
$P\mapsto {\rm up}(P)$ on a collection of paths $P=(p_1,\ldots,p_l)$ satisfying conditions (1)-(3) above. 

\begin{arule}\label{uprule} $\,$\vspace{-2mm}
\begin{itemize}
\item[-] Locate the leftmost up step in $p_1$, and let $\overline{y}$ be its label. Replace it with a horizontal 
step at height $y$ and a down step below it ending at height $0$. 
\item[-] Shift up by $1$ the segment of $p_1$ to the left of the position where the above change occurred, 
to connect it to the tail of $p_1$. Shift up by $1$ all the other paths. 
\item[-] Replace any down step above height $n$ by a down step below it, ending at height $0$. 
\item[-] Consider the down steps ending at height $0$ from right to left, excluding the rightmost one. 
Match them with the shifts of $p_2,\ldots,p_l$, in this order, and connect the matched pairs by horizontal 
lines of height $1$. 
\end{itemize}
\end{arule}

Note that, in the last step of the rule, the last one or two down steps might have no match, so they start 
new paths. For simplicity, any horizontal steps at the beginning of a path are ignored. It is easy to see that, 
since $P$ satisfies conditions (1)-(3) above, the rule can be applied, and ${\rm up}(P)$ satisfies the same 
conditions. Thus, we can recursively define $P_{j+1}:={\rm up}(P_j)$ as long as there are up steps in $P_j$. 

We claim that applying the rule to $P_j$ corresponds to applying $F_j$ to $v_j$. To make this precise, 
we introduce some notation. Let $r$ be the collection of columns of $B=B^{r_N,1}\otimes \cdots \otimes B^{r_1,1}$, i.e.,
the tuple $(r_N,\ldots,r_1)$. Given a word $w$ of length $r_1+\cdots+r_N$, we define a rearrangement of it 
${\rm ord}_r(w)$ by slicing $w$ into segments of length $r_N,\ldots,r_1$ in this order, and by 
reordering each segment decreasingly. We will show below that, for all $j=0,\ldots, s$, we have
\begin{equation}\label{words-agree}
{\rm ord}_r({\rm word}(P_{j}))={\rm word}(v_{j})\,.
\end{equation}
We will see that the reordering of the segments of ${\rm word}(P_{j})$ done by ${\rm ord}_r$, if any, is 
very simple: a segment is a concatenation of two decreasing segments which can be swapped to give 
a decreasing sequence. Based on the above discussion, one can see that $\lambda_j$ is the weight of 
$v_j$, and that this element is of highest weight in its classical (non-affine) component.
  
\begin{example}\label{expaths}{\rm The collections of paths $P_j$ corresponding to 
Example~\ref{example.paths} is shown in Figure~\ref{figure.paths}.
\begin{figure}
\begin{equation*}
\begin{split}
v_0 &= \tableau{2} \otimes \tableau{2\\ \overline{2}} \otimes \tableau{\overline{3}\\ \overline{2}} \otimes
\tableau{1\\ 2\\ 3}
\qquad 
\raisebox{-1.5cm}{
\begin{picture}(200,100)(0,0)
	\put(180,20){\line(-1,1){60}}
	\put(120,80){\line(-1,-1){40}}
	\put(80,40){\line(-1,1){20}}
	\put(60,60){\line(-1,-1){20}}
	\put(40,40){\line(-1,1){20}}
	\put(175,30){1} \put(155,50){2} \put(135,70){3}
	\put(100,70){$\overline{3}$} 
	\put(86,55){$\overline{2}$} \put(67,55){2} 
	\put(46,55){$\overline{2}$} \put(27,55){2} 
\end{picture}
}
\qquad
\begin{array}{l} \text{weight of path}\\ \tableau{\mbox{}\\ \mbox{}} \end{array}\\
v_1 &= \tableau{3} \otimes \tableau{1\\ 2} \otimes \tableau{\overline{3}\\ \overline{2}} \otimes
\tableau{1\\ 2\\ 3}
\qquad 
\raisebox{-1.5cm}{
\begin{picture}(200,100)(0,0)
	\put(180,20){\line(-1,1){60}}
	\put(120,80){\line(-1,-1){40}}
	\put(80,40){\line(-1,1){20}}
	\put(60,60){\line(-1,0){20}}
	\put(40,60){\line(-1,1){20}}
	\put(60,20){\line(-1,1){20}}
	\put(175,30){1} \put(155,50){2} \put(135,70){3}
	\put(100,70){$\overline{3}$} 
	\put(86,55){$\overline{2}$} \put(67,55){2} \put(35,70){3}
	\put(47,35){1}
\end{picture}
}
\qquad
\tableau{\mbox{}&\mbox{}\\ \mbox{} \\ \mbox{}} \\
v_2 &= \tableau{1} \otimes \tableau{2\\ 3} \otimes \tableau{1\\ \overline{3}} \otimes
\tableau{1\\ 2\\ 3}
\qquad 
\raisebox{-1.5cm}{
\begin{picture}(200,100)(0,0)
	\put(180,20){\line(-1,1){60}}
	\put(120,80){\line(-1,-1){20}}
	\put(100,60){\line(-1,0){20}}
	\put(80,60){\line(-1,1){20}}
	\put(100,20){\line(-1,1){20}}
	\put(80,40){\line(-1,0){20}}
	\put(60,40){\line(-1,1){20}}
	\put(40,20){\line(-1,1){20}}
	\put(175,30){1} \put(155,50){2} \put(135,70){3}
	\put(100,70){$\overline{3}$} \put(74,70){3} 
	\put(95,30){1} \put(55,50){2} 
	\put(35,30){1}
\end{picture}
}
\qquad
\tableau{\mbox{}&\mbox{}&\mbox{}\\ \mbox{}&\mbox{}\\ \mbox{}}\\
v_3 &= \tableau{2} \otimes \tableau{1\\ 3} \otimes \tableau{1\\ 2} \otimes
\tableau{1\\ 2\\ 3}
\qquad 
\raisebox{-1.5cm}{
\begin{picture}(200,100)(0,0)
	\put(180,20){\line(-1,1){60}}
	\put(120,20){\line(-1,1){40}}
	\put(80,60){\line(-1,0){20}}
	\put(60,60){\line(-1,1){20}}
	\put(80,20){\line(-1,1){20}}
	\put(60,40){\line(-1,0){20}}
	\put(40,40){\line(-1,1){20}}
	\put(175,30){1} \put(155,50){2} \put(135,70){3}
	\put(115,30){1} \put(95,50){2} \put(55,70){3}
	\put(75,30){1} \put(35,50){2}
\end{picture}
}
\qquad
\tableau{\mbox{}&\mbox{}&\mbox{}\\ \mbox{}&\mbox{}& \mbox{} \\ \mbox{}&\mbox{}}
\end{split}
\end{equation*}
\caption{Paths for Example~\ref{expaths}
\label{figure.paths}}
\end{figure}
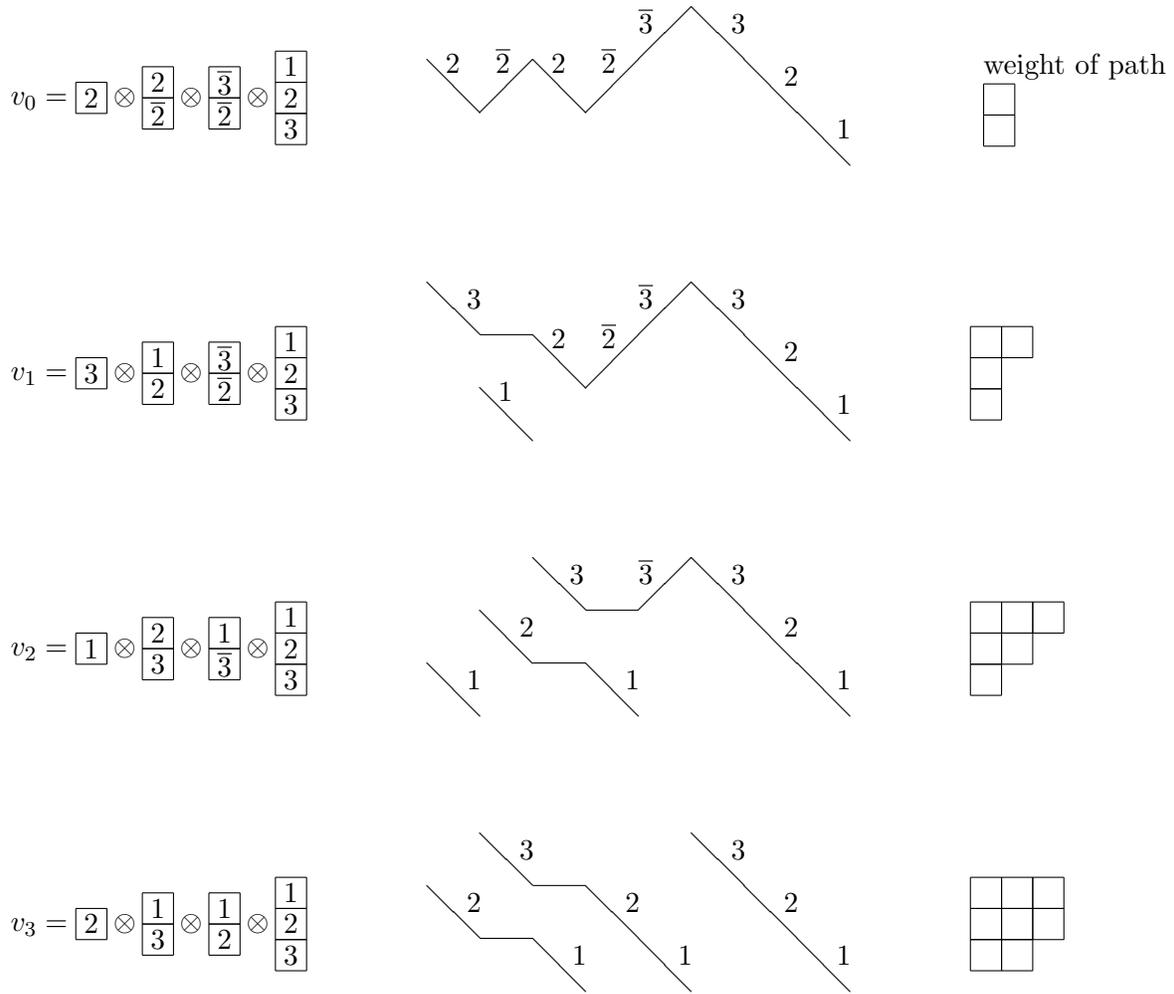

}
\end{example}

\begin{proof}[Proof of Theorem~{\rm \ref{theorem.ground2typeA}}]
It essentially suffices to show that $F_j(v_j)$ is defined if and only if ${\rm up}(P_j)$ is defined and that 
\eqref{words-agree} holds, for all $j$ (the only extra fact to check is that all arrows corresponding to $F_j$ 
are Demazure arrows, see below). Indeed, Rule~\ref{uprule} can be applied as long as there are up steps 
in $P_j$, so the final element $v_s$ will have only positive entries.  
We will first show that Rule~\ref{uprule} precisely describes the action of $F_j$ on the word of $P_j$, viewed 
as an element of the tensor product $(B^{1,1})^{\otimes (r_1+\cdots+r_N)}$; more precisely, we have
\begin{equation}\label{f-on-paths}
F_j({\rm word}(P_j))={\rm word}(P_{j+1})\,.
\end{equation}
We will then show that the above action of $F_j$ is completely similar to that on $v_j$, which will 
prove~\eqref{words-agree}.

Let us describe the action of $F_j$ on ${\rm word}(P_j)$. First note that the connected components of the 
paths in $P_j$ correspond to bracketed units in the tensor product $(B^{1,1})^{\otimes (r_1+\cdots+r_N)}$
according to the signature rule for the application of the Kashiwara operator on tensor products of crystals.
This means that any given operator $f_i$ can only operate on the starting points of the connected paths in
$P_j$. The application of $f_n^k$ lifts all starting points of value $n$ to $\overline{n}$.
Note that by the weight $\lambda_j=\lambda(k,h_2,h_1)$ of $v_j$ there are precisely $k$ of them. Next $f_{n-1}^{2k}$ lifts all $k$
just created $\overline{n}$ to $\overline{n-1}$ and the $k$ letters $n-1$ to the left of the leftmost up-step to $n$ that were previously
bracketed with the just lifted $n$ etc.. Potentially there are two more connected path components with starting points
at $h_1$ and $h_2$, hence the exponents of $f_i$ with $i\le h_1$ and $i\le h_2$ are increased accordingly. 
In summary, the starting points $n$ of path components are eventually lifted to $\overline{1}$ and then turned into $1$ by
$f_0$; the other positive letters to the left of the leftmost up step are all raised by one. In the rightmost path, this lifting
process eventually reaches the leftmost up step $\overline{y}$ which is lifted to $\overline{y-1}$ by $f_{y-1}$.
The steps to the right of the leftmost up step $\overline{y}$ are either $\overline{y+1}$ or $y$. Both would need an $f_y$ to be 
lifted. Since the $f_i$ in $F_j$ are applied in decreasing order of the indices $i$, this cannot happen and hence the leftmost up step
is lifted to $\overline{1}$ and then turned into $1$ by $f_0$. Altogether, the changes are precisely as described in
Rule~\ref{uprule}.

Let us now compare the action of $F_j$ on ${\rm word}(P_j)$, i.e., on $(B^{1,1})^{\otimes (r_1+\cdots+r_N)}$, 
with that on $v_j$, i.e., on $B=B^{r_N,1} \otimes \cdots \otimes B^{r_1,1}$. Note that by the action of $f_0$ on columns
rather than tensor products of single boxes, the word ${\rm word}(P_j)$ of a path is a concatenation of
cyclically shifted columns of $v_j$. We claim that $F_j$ acts on 
precisely the same entries in the two cases. In the expression for $F_j$ the $f_i$ with larger
indices $i$ act first and hence lift up the upper or barred portion of the paths. This is the case before
or after reordering. $f_0$ changes a $\bar{1}$ into a $1$ in both cases. Hence the action commutes
with the reordering.

Finally, the rightmost tensor factor in $v_0$ is the column $c=r_1\cdots 1\in B^{r_1,1}$. Note that this
column is never changed during the algorithm and satisfies $\varepsilon_0(c)=1$. By the tensor product
rules this implies that $\varepsilon_0(v_j)\ge 1$ for all $0\le j<s$, so that all arrows are Demazure arrows.
\end{proof}

There is a more direct way of constructing the type $A$ elements $b$ in 
Theorem~\ref{theorem.ground2typeA} from the ground state $u\otimes u_{\Lambda_0}$. For $i=1,2,\ldots,N$, 
place the letter $i$ in column 1 for each unbarred letter in $u_i$ and in column 2 for each barred letter, 
where $u=u_N \otimes \cdots \otimes u_1$. Note that, due to the fact that $u$ is a ground state, the difference
in height between the first and second column is at most $n$. Now cut the columns at heights 
$n, 2n, 3n,\ldots$ and put all pieces next to each other aligned at height 0. The letters in row $i$ record 
the tensor factors of $b$ in Theorem~\ref{theorem.ground2typeA} which contain the letter $i$. Comparing 
this with the algorithm in the proof of Theorem~\ref{theorem.ground2typeA}, it is not hard to see that it gives 
the correct answer.

\begin{example}
{\rm
Continuing Examples~\ref{example.paths} and \ref{expaths}, the chosen ground state yields the columns
\[
\tableau{{4}\\ {3} \\ {1}&{3}\\ {1}&{2}\\ {1}&{2}}
	\qquad \text{and, after the cut,} \qquad
	\tableau{{1}&{3}\\ {1}&{2}&{4}\\ {1}&{2}&{3}} \; .
\]
This tells us that the letter 1 appears in $b_1$, $b_2$, $b_3$, the letter 2 appears in $b_1$, $b_2$, $b_4$,
and the letter 3 appear in $b_1$ and $b_3$. This agrees with Figure~\ref{figure.paths}.
}
\end{example}

%%%%%%%%%%%%%%%%%%%%%%%%%%%%%%%%%%%%%%%%%%%%%%%%%%%%%%%%%%%%%%%
\section{Energy and charge in type $C$}
\label{section.energyC}
%%%%%%%%%%%%%%%%%%%%%%%%%%%%%%%%%%%%%%%%%%%%%%%%%%%%%%%%%%%%%%%

The purpose of this section is to provide the proof of Theorem~\ref{theorem.main} for type $C$.

%%%%%%%%%%%%%%%%%%%%%%%%%%%%%%%%%%%%%%%%%%%%%%%%%%%%%%%%%%%%%%%
\subsection{The energy for type $A$ fillings}\label{energya}

We start with a brief discussion of the combinatorial $R$-matrix $B_1\otimes B_2\rightarrow B_2\otimes B_1$
for a tensor product of two type $C$ columns. By the type $C$ Pieri rule \cite{Su:1990}, the decomposition of 
$B_1\otimes B_2$ into classical (non-affine) components is multiplicity free. On the other hand, it was proved
in~\cite{lecsc} that the type $C$ jeu de taquin due to Sheats \cite{shesjt} is compatible with the classical 
crystal operators. We conclude that the mentioned jeu de taquin, when applied to our two-column situation, 
realizes the corresponding combinatorial $R$-matrix. Note that, in all situations considered below, the type 
$C$ jeu de taquin is governed by the simpler rules of the classical one, in type $A$ (see, e.g., \cite{fulyt}). 

\begin{lemma}\label{localac}
Let $b=b_1\otimes b_2$ be a tensor product of two columns with all entries in $[n]$, where the height 
of $b_2$ is at most that of $b_1$. Then the type $A_{n-1}^{(1)}$ and $C_n^{(1)}$ (local) energies of $b$ coincide.
\end{lemma} 

\begin{proof}
Consider the crystal $B^{h_1,1}\otimes B^{h_2,1}$ of type $C_n^{(1)}$, where $n\ge h_1\ge h_2\ge 1$. By the
type $C$ Pieri rule, the classical components of this crystal containing fillings with all entries in $[n]$ have 
highest weights $\mu_i=(h_1+i,h_2-i)'$, for $i=0,\ldots,\min(h_2,n-h_1)$. We will calculate the (local) energy 
on these components based on its definition (\ref{eq:local energy}), by setting the energy to $0$ on the
component of highest weight $\mu_0$ (the usual normalization). 

For $i=1,\ldots,\min(h_2,n-h_1)$, consider $b_{1i}\otimes b_{2i}$ in $B^{h_1,1}\otimes B^{h_2,1}$, where
\[b_{1i}=\begin{array}{|c|}\hline n-h_1+1 \\ \hline  \vdots \\ \hline n 
\\ \hline \end{array}\,,\qquad
b_{2i}=\begin{array}{|c|}\hline  
1\\ \hline n-h_1-i+2 \\ \hline  \vdots \\ \hline n-h_1 \\ \hline n-h_2+1 \\ \hline \vdots \\ \hline n-i \\ \hline \end{array}\,.\]
The image $b_{2i}'\otimes b_{1i}'$ of $b_{1i}\otimes b_{2i}$ under the combinatorial $R$-matrix, constructed with jeu de taquin, is given by 
\[b_{2i}'=\begin{array}{|c|}\hline n-h_2+1 \\ \hline  \vdots \\ \hline n 
\\ \hline \end{array}\,,\qquad
b_{1i}'=\begin{array}{|c|}\hline  
1\\ \hline n-h_1-i+2 \\ \hline  \vdots \\ \hline  n-i \\ \hline \end{array}\,.\]
Clearly, $e_0$ acts RR on $b_{1i}\otimes b_{2i}$ and $b_{2i}'\otimes b_{1i}'$, by changing $1$ to 
$\overline{1}$. By type $C$ insertion \cite{lecsc}, which in this case is essentially just type $A$ insertion 
(see, e.g., \cite{fulyt}), it is easy to see that $b_{1i}\otimes b_{2i}$ and $e_0(b_{1i}\otimes b_{2i})$ lie in 
the components of highest weights $\mu_i$ and $\mu_{i-1}$, respectively. Therefore, the energy on the component of highest weight $\mu_i$ is $-i$. This coincides with the type $A$ energy (calculated via a 
similar procedure or via the type $A$ charge). 
\end{proof}

\begin{proposition}\label{chc2}
If $b$ is a tensor product of columns with all entries in $[n]$, then the type $A_{n-1}$ and $C_n$ energies 
of $b$ coincide. Furthermore, if the columns have weakly decreasing heights, they equal the 
(type $A_{n-1}$ or $C_n$) charge of $b$.
\end{proposition}

\begin{proof} When all entries are in $[n]$, the jeu de taquin algorithms in types $A_{n-1}$ and $C_n$ 
(realizing the corresponding combinatorial $R$-matrices) work identically. Since the corresponding local 
energies coincide by Lemma~\ref{localac}, the global energies coincide as well. But the type $A_{n-1}$ 
energy of $b$ is computed by the type $A_{n-1}$ charge when the heights of the columns are weakly
decreasing by Theorem~\ref{theorem.main} which was proven in Section~\ref{section.energyA}, 
which clearly coincides with the corresponding type $C_n$ charge.
\end{proof}

%%%%%%%%%%%%%%%%%%%%%%%%%%%%%%%%%%%%%%%%%%%%%%%%%%%%%%%%%%%%%%%
\subsection{The conclusion of the proof}\label{endproof}

We start by studying the behavior of the type $C$ charge with respect to the crystal operators and state a result in \cite{lalgam}.

\begin{proposition}\label{chcf1n}\cite{lalgam}
The type $C_n$ charge is preserved by the crystal operators $f_1,\ldots,f_{n}$.
\end{proposition}

\begin{proposition}\label{chce0} 
Let $B=B^{r_N,1}\otimes \cdots \otimes B^{r_1,1}$ be of type $C_n$ with $r_N\ge r_{N-1} \ge \cdots \ge
r_1$ and $b\in B$.
If $\varphi_0(b)\ge 1$ and $\varepsilon_0(b)\ge 1$, then the type $C_n$ charge satisfies 
\[\charge(e_0(b))=\charge(b)-1\,.\]
\end{proposition}

\begin{proof}
Let $b=b_1\ldots b_{\mu_1}$, and note that none of these columns contain both $1$ and $\overline{1}$. 
Assume that $e_0$ changes the entry $1$ in $b_j$ to $\overline{1}$. The condition $\varphi_0(b)\ge 1$ 
implies $j>1$. Note that the column $b_{j-1}$, and thus $b_{j-1}^R$, cannot contain $1$ (otherwise, the 
entry $1$ in $b_j$ would not be the leftmost unpaired $1$). Similarly, the column $b_{j+1}$, and thus 
$b_{j+1}^L$, if they exist, cannot contain $\overline{1}$ (otherwise, the entry $1$ in $b_j$ would be 
paired with the mentioned $\overline{1}$). 

Let $c:=\mbox{circ-ord}(b)$ with $c=c_1^Lc_1^R\ldots c_{\mu_1}^Lc_{\mu_1}^R$, and $d:=\mbox{circ-ord}(e_0(b))$ with 
$d=d_1^Ld_1^R\ldots d_{\mu_1}^Ld_{\mu_1}^R$ (recall that the map $\mbox{circ-ord}$ is defined by a slight modification 
of Algorithm \ref{algreord}). Assume that $c_j^L(i)=1$, which implies $c_j^R(i)=1$ (as we have no descents 
in the pair $c_j^Lc_j^R$; see Section \ref{constrc}). We claim that $d_l^L(k)=c_l^L(k)$ and $d_l^R(k)=c_l^R(k)$ 
for all $k,l$ with the exception of $d_j^L(i)=d_j^R(i)=\overline{1}$. Indeed, all we need to check are the following, 
in this order: (1) the entries $d_j^L(k)$ and $d_j^R(k)$ for $k=1,\ldots,i-1$ are as claimed, since the alternative,
namely that one of them is $\overline{1}$, would lead to a contradiction; (2) the value of $d_j^L(i)$ follows from 
the fact that $b_{j-1}^R$ contains no $1$; (3) the value of $d_{j+1}^L(i)$, if this entry exists, follows from the fact 
that $b_{j+1}^L$ contains no $\overline{1}$. 

By the above facts, we have the same descents in the fillings $c$ and $d$, with the exception of the descent 
$c_{j-1}^R(i)>c_j^L(i)=1$, which corresponds to $d_j^R(i)=\overline{1}>d_{j+1}^L(i)$, assuming that 
$d_{j+1}^L(i)$ exists. The difference between the arm lengths of the mentioned descents is $2$, which 
concludes the proof by~\eqref{chdesc}.
\end{proof}

\begin{proof}[Proof of Theorem {\rm \ref{theorem.main}} in type $C$]
The proof is immediate based on Lemma~\ref{rec-const-energy}, Theorem~\ref{theorem.ground2typeA}, and 
Propositions~\ref{chc2}, \ref{chcf1n}, \ref{chce0}.
\end{proof}

%%%%%%%%%%%%%%%%%%%%%%%%%%%%%%%%%%%%%%%%%%%%%%%%%%%%%%%%%%%%%%%
\section{Open problems}
\label{section.outlook} 
%%%%%%%%%%%%%%%%%%%%%%%%%%%%%%%%%%%%%%%%%%%%%%%%%%%%%%%%%%%%%%%

In this section we discuss several directions of research stemming from the results in this paper. We intend to 
pursue some of them in the future.

%%%%%%%%%%%%%%%%%%%%%%%%%%%%%%%%%%%%%%%%%%%%%%%%%%%%%%%%%%%%%%%
\subsection{Columns of types $B$ and $D$} 
\label{section.columnsBD}
We believe that, in types $B_n$ and $D_n$, the statistic in the Ram--Yip formula for Macdonald polynomials at $t=0$ can also be 
translated into a charge statistic on the corresponding tensor product of KR crystals $B^{k,1}$ of types $B_n^{(1)}$
and $D_n^{(1)}$ (represented with KN columns), 
thus extending the results in~\cite{lenmpc}. We conjecture that the type $B_n$ and $D_n$ charge also agrees with the 
corresponding energy function, which would extend the results in this paper.  It should be possible to use a proof technique similar 
to the one in this paper, based on Lemma~\ref{rec-const-energy} from~\cite{ST:2011}, to prove the conjecture; 
the proof would include the generalization of the results in~\cite{lalgam} to types $B_n$ and $D_n$. These claims are supported 
by~\cite[Corollary 9.5]{ST:2011} expressing a type $D_n$ Macdonald polynomial at $t=0$ in terms of the corresponding 
energy function. 

From one point of view, the case of types $B_n$ and $D_n$ is easier than the one of type $C_n$, because the corresponding 
level 1 KR crystals are perfect. However, the construction of charge in types $B_n$ and $D_n$ displays additional complexity, 
due to some new aspects, that we now describe. We start by referring to type $B_n$, as type $D_n$ has all the complexity 
of type $B_n$ plus an additional one. For the corresponding KN columns, indexing the vertices of the crystals $B(\omega_i)$ 
corresponding to the fundamental representations $V(\omega_i)$, we refer to \cite{kancgr}. 

The first new aspect in type $B_n$ is related to the splitting of the KR crystal $B^{k,1}$ upon removing the 0-arrows as the 
following direct sum of (classical) crystals: $B(\omega_{k})\oplus B(\omega_{k-2})\oplus\cdots$. This phenomenon 
manifests itself in the existence of descents between the left and the right columns of a split KN column, which was not the 
case in type $C_n$. We illustrate this based on the following example.

\begin{example}
{\rm Let $\mu=\omega_4$ in $B_5$. The KR crystal $B^{4,1}$ contains a vertex indexed by the KN column
\[
	b=\tableau{{3}\\{\overline{3}}}\;\in\; B(\omega_2)\subset B^{4,1}\simeq 
	B(\omega_4)\oplus B(\omega_2)\oplus B(\omega_0)\,.
\]
Instead of constructing the split column, we first construct from $b$ the following ``extended'' KN column of height 4, 
in order to make the columns corresponding to the components of $B^{4,1}$ of the same height:
\[
	\widehat{b}=\tableau{{1}\\{3}\\{\overline{3}}\\{\overline{1}}}\,;
\]
this construction is based on a procedure in~\cite[Section 3.4]{schbtd}. Only at this point we construct the split column. 
We claim that the doubling procedure in Definition~\ref{defkn} can be extended, but we still match the entries in $I=\{3>1\}$ 
with certain entries ``preceding'' them; the only difference is that this now requires us to go counterclockwise around the 
circle. So we obtain  $J=\{2,\overline{4}\}$, which gives the following splitting of $\widehat{b}$ by the usual rule:
\[
	\widehat{b}^L \widehat{b}^R=\tableau{{2}&{1}\\{\overline{4}}&{3}\\{\overline{3}}&{4}\\{\overline{1}}&{\overline{2}}}\,.
\]
Finally, we construct $c=c^L c^R=\mbox{circ-ord}(\widehat{b}^L \widehat{b}^R)$ by the usual rule (see Section \ref{constrc}), where 
$c^L=\widehat{b}^L$:
\[
	c=\tableau{{2}&{3}\\{\overline{4}}&{\overline{2}}\\{\overline{3}}&{1}\\{\overline{1}}&{4}}\,.
\]
The energy of $b$ in $B^{4,1}$ is $-1$. This agrees with the charge of $\widehat{b}^L \widehat{b}^R$, computed via the sum 
in~\eqref{chdesc}; indeed, there are two descents in $c$, whose arm lengths are $1$, so 
${\rm charge}(\widehat{b}^L \widehat{b}^R)=\frac{1}{2}(1+1)=1$.
}
\end{example}

The second new aspect in type $B_n$ is the fact that not always a descent in the filling obtained via the procedure 
``$\mbox{circ-ord}$'' contributes half its arm length to the charge. We illustrate this with another example.

\begin{example}\label{exxx}
{\rm Let $\mu=2\omega_2$ in type $B_3$, so $B_\mu=B^{2,1}\otimes B^{2,1}$. Consider the pair of KN columns
\[
	b_1b_2=\tableau{{1}&{1}\\{\overline{2}}&{2}}
\]
in $B_\mu$. The corresponding fillings with split columns $b=b_1^Lb_1^Rb_2^Lb_2^R$ and 
$c=\mbox{circ-ord}(b)=c_1^L c_1^R c_2^L c_2^R$ coincide, and they are
\[
	\tableau{{1}&{1}&{1}&{1}\\{\overline{2}}&{\overline{2}}&{2}&{2}}\,.
\]
The corresponding energy is $-2$, but in $c$ we have only one descent of arm length 2, so we cannot use the sum 
in~\eqref{chdesc} to compute the charge. This problem occurs because, for all the entries $i$ between $\overline{2}$ and $2$ 
in circular order (namely $i\in\{\overline{1},1\}$), either $i$ or $\overline{\imath}$ are above $\overline{2}$ in $c_1^R$. This 
forces the descent $\overline{2}>2$ to be double counted. A more illuminating explanation can be given in the context of 
deriving the charge from the statistic in the Ram--Yip formula; see~\cite[Section 9]{lenmpc}.
}
\end{example}

The above example suggests that in type $B_n$ we need to modify the definition of charge by the sum in~\eqref{chdesc} 
as follows. We use the same notation $c=c_1^L c_1^R \cdots c_{\mu_1}^L c_{\mu_1}^R$ as above. Let ${\rm Des}'(c)$ denote 
the descents of the form $\overline{m}=c^R_j(i)>c^L_{j+1}(i)=m$ such that, for any $k=1,\ldots,m-1$, we have either $k$ or 
$\overline{k}$ in $c^R_j[1,i-1]$. We claim that the appropriate definition of the type $B_n$ charge is given by the following 
formula:
\begin{equation}\label{bcharge}
	\frac{1}{2}\sum_{\gamma\in{\rm Des}(c)\setminus{\rm Des}'(c)}{\rm arm}(\gamma)+
	\sum_{\gamma\in{\rm Des}'(c)}{\rm arm}(\gamma)\,.
\end{equation}

We said above that in type $D_n$ we have additional complexity still. We explain this using another example. 

\begin{example}
{\rm Let $\mu=2\omega_1$ in type $D_3$. 
The filling with split columns $c=b=b_1^L b_1^R b_2^L b_2^R=\tableau{{3}&{3}&{\overline{3}}&{\overline{3}}}$ in 
$B_\mu$ has no descents, but the corresponding energy is $-1$. The reason for this is that the values $3$ and 
$\overline{3}$ are incomparable in type $D_3$, so the pair $\tableau{{3}&{\overline{3}}}$ in $c$ needs to contribute 
$1$ to the charge. In fact, the definition of charge needs to be adjusted even more, as the following more subtle example 
shows.

Let $\mu=2\omega_2$ in $D_4$. 
Consider the following filling with split columns in $B_\mu$:
\[
	c=b=b_1^L b_1^R b_2^L b_2^R
	=\tableau{{3}&{3}&{\overline{4}}&{\overline{4}}\\{\overline{4}}&{\overline{4}}&{\overline{3}}&{\overline{3}}}\,.
\]
This filling has only ascents or equal entries next to each other in a row. However, the corresponding energy is $-1$, so the 
charge needs to be 1. The reason for which the charge is not 0, meaning that $b$ is not a split KN tableau, is that the two 
middle columns in $b$ represent a forbidden configuration in type $D_4$. (Recall from \cite{kancgr} that, in type $D$, in 
addition to the row monotonicity condition for the split tableau in types $B$ and $C$, there are extra conditions for a sequence 
of KN columns to form a KN tableau, and these are given in terms of certain forbidden configurations for a pair formed by a 
right column and the next left column.) 
}
\end{example}

%%%%%%%%%%%%%%%%%%%%%%%%%%%%%%%%%%%%%%%%%%%%%%%%%%%%%%%%%%%%%%%
\subsection{Rows of types $B$, $D$, and $C$} 
In type $A$, Nakayashiki and Yamada~\cite{NY:1997} showed that the Lascoux--Sch\"{u}tzenberger charge 
statistic is related to the energy function on tensor products of both rows and columns. 
In particular, in~\cite[Proposition 3.23]{NY:1997} they provide an explicit relation between the energy function
on rows and columns by giving a bijection of each set to the set of semistandard Young tableaux which
preserves the statistics.

We expect a similar result to hold in the other classical types (with slight modifications). 
This claim is motivated by~\cite[Theorem 10.10]{LOS:2010}, which relates the one-dimensional configuration
sum in the large rank limit for columns to the one for rows in the dual type. More explicitly, denote by
$X_{\lambda, B}^{Y_n}(q)$ the one-dimensional configuration sum, that is, the sum of $q^{D(b)}$ over all highest weight
elements $b$ in the crystal $B$ of weight $\lambda$ and type $Y_n$, graded by the energy $D(b)$. 
Then up to an overall power of $q$, $X_{\lambda, B}^{B_n}(q)$ and $X_{\lambda', B'}^{C_n}(q^{-1})$ are equal
for large rank $n$. Here, if $B=B^{r_1,s_1}\otimes \cdots \otimes B^{r_L,s_L}$ then
$B' = B^{s_1,r_1} \otimes \cdots \otimes B^{s_L,r_L}$ and $\lambda'$ is the transpose of $\lambda$ (where we
identify dominant weights with partitions). Note that $X_{\lambda, B}^{B_n}(q) =X_{\lambda, B}^{D_n}(q)$
when $n$ is large~\cite{LOS:2010,SZ:2006}.

In particular, if the charge for columns of types $B_n$ or $D_n$ were known as outlined in Section~\ref{section.columnsBD}
then, using $B=B^{r_1,1}\otimes \cdots \otimes B^{r_L,1}$, the relation between $X_{\lambda, B}^{B_n}(q)$ and 
$X_{\lambda', B'}^{C_n}(q^{-1})$ would provide a way to obtain the charge in type $C_n$ also for single rows 
(as opposed to single columns as treated in this paper).
Computer experiments using {\sc Sage}~\cite{sage} indicate that this duality not only holds on the level of
configuration sums but also on the level of individual terms. We expect that the highest weight crystal elements 
for both columns and rows are in bijection with certain oscillating tableaux, as introduced in~\cite{Berele:1986}, due to the fact that
the recording tableaux for the Robinson--Schensted algorithm for types $B_n,C_n,D_n$ as introduced 
in~\cite{lecsc,Lecouvey:2007} are given by these tableaux.
If this conjecture were to be true, then we would be able to obtain the charge for single rows from the charge for single columns
via this bijection in the stable range (that is, large $n$).

%%%%%%%%%%%%%%%%%%%%%%%%%%%%%%%%%%%%%%%%%%%%%%%%%%%%%%%%%%%%%%%
\subsection{Arbitrary order of columns and arbitrary types}\label{generalize} 
Consider a composition $r=(r_N,\ldots,r_1)$, and the crystal $B:=B^{r_N,1}\otimes\cdots\otimes B^{r_1,1}$ in types $A$ and $C$. 
Since the charge is defined only if $r$ is a partition, the computation via the charge of the energy function on some vertex $b$ in 
$B$ must proceed in general indirectly, by first appropriately commuting the columns of $b$ via the combinatorial $R$-matrix 
(which preserves the energy). However, it is desirable to compute the energy directly on $B$. In addition, we would like to 
generalize to $B$ of arbitrary type, i.e., to find a statistic on $B$ expressing the energy, which can be computed using only 
the combinatorial data corresponding to a vertex, instead of various paths in $B$. 

Let $\mu(r)$ be the partition obtained by reordering the parts of $r$. The Ram--Yip formula for the Macdonald polynomial 
$P_{\mu(r)}(x;q,0)$ can be written in a way that is compatible with the composition $r$ rather than the partition $\mu(r)$. 
Recall from~\cite[Proposition 2.7]{lenmpc} that the terms in the above formula (for arbitrary type) correspond to chains of Weyl 
group elements giving rise to paths in the quantum Bruhat graph~\cite{bfpmbo}; 
these chains are partitioned into segments corresponding to the parts of $r$, and the down steps (in Bruhat order) 
are measured by a statistic called {\em level}. In \cite{lalgam} we defined crystal operators on the above chains, both 
classical ones $f_i$, $i>0$, and the affine one $f_0$. We called this construction the {\em quantum alcove model}, as 
it generalizes the alcove model in \cite{lapawg,lapcmc}. The main conjecture is that the new model uniformly describes 
tensor products of column shape KR crystals, for all untwisted affine types. The conjecture is proved in the same paper 
in types $A$ and $C$, by showing that the bijections in \cite{lenmpc} from the objects of the quantum alcove model to 
tensor products of the corresponding KN columns are affine crystal isomorphisms. With S. Naito and M. Shimozono, 
we are working on a uniform proof of the conjecture for all untwisted affine types.  This will imply that the level statistic 
in the Ram-Yip formula expresses the corresponding energy function. In particular, we would obtain a generalization 
of the combinatorial formula (\ref{p1}) to arbitrary type, where the charge is replaced by the level.

%%%%%%%%%%%%%%%%%%%%%%%%%%%%%%%%%%%%%%%%%%%%%%%%%%%%%%%%%%%%%%%
\subsection{Demazure crystals and Macdonald polynomials} 
Ion~\cite{ionnmp} showed that, in simply-laced types, one can identify a Macdonald polynomial at $t=0$, namely $P_\mu(x;q,0)$, 
with an affine Demazure character corresponding to a shift in the affine Weyl group. This raises the question of what happens in 
other types. 
The combinatorial formula (\ref{p1}) for $P_\mu(x;q,0)$ of type $C$, together with the decomposition of the corresponding 
tensor product of KR crystals with only the Demazure arrows (studied in Section \ref{section.kyotoC}) suggest that, in general, 
$P_\mu(x;q,0)$ could be identified with a sum of Demazure characters. If so, one would want to determine them explicitly and, 
in particular, to know if all correspond to shifts in the affine Weyl group. 

As Ion's result and the Ram--Yip formula were given for non-symmetric Macdonald polynomials as well, we can also ask 
about the generalization of Ion's result to arbitrary types in the non-symmetric case. There are indications that the 
combinatorial formula~\eqref{p1}, which was conjectured to have a version for arbitrary type (see Section~\ref{generalize}), 
should be replaced with one whose right-hand side is a sum over a subset of the corresponding tensor product of KR crystals. 
If so, it would be desirable to describe this subset in an explicit way. See also~\cite[Section 9.2]{ST:2011}.

%%%%%%%%%%%%%%%%%%%%%%%%%%%%%%%%%%%%%%%%%%%%%%%%%%% 

\end{document}